\documentclass[a4paper,12pt,final]{amsart}
\usepackage{times,a4wide,mathrsfs,amssymb,dsfont, enumerate}

\newcommand{\C}{\mathbb{C}}
\newcommand{\Z}{\mathbb{Z}}

\newcommand{\QQ}{\mathbb{Q}}
\newcommand{\NN}{\mathbb{N}}
\newcommand{\PP}{\mathbb{P}}

\newcommand{\OO}{\mathcal O}
\newcommand{\Ss}{\mathcal S}

\newcommand{\XX}{\mathcal X}
\newcommand{\YY}{\mathcal Y}
\newcommand{\CC}{\mathcal C}

\newcommand{\VV}{\mathcal V}
\newcommand{\WW}{\mathcal W}
\newcommand{\MM}{\mathcal M}
\newcommand{\FF}{\mathcal F}
\newcommand{\EE}{\mathcal E}

\newcommand{\ZZZ}{\mathcal Z}
\newcommand{\hh}{\mathfrak h}
\newcommand{\ttt}{\mathfrak t}

\newcommand{\wt}{\widetilde}
\newcommand{\ima}{\hbox{Im}}
\newcommand{\rom}{\romannumeral}

\newcommand{\ide}{\hbox{id}}

\newtheorem{theorem}{Theorem}[section]
\newtheorem{claim}[theorem]{Claim}
\newtheorem{lemma}[theorem]{Lemma}
\newtheorem{corollary}[theorem]{Corollary}
\newtheorem{proposition}[theorem]{Proposition}

\newtheorem{nonumbering}{Theorem}

\newtheorem{nonumberingc}{Corollary}

\newtheorem{convention}{Conventions}

\theoremstyle{definition}
\newtheorem{remark}[theorem]{Remark}
\newtheorem{definition}[theorem]{Definition}

\newtheorem{notation}[theorem]{Notation}

\newtheorem{nonumberingt}{Acknowledgements}

\begin{document}
\author[Robert Laterveer]
{Robert Laterveer}

\address{Institut de Recherche Math\'ematique Avanc\'ee,
CNRS -- Universit\'e 
de Strasbourg,\
7 Rue Ren\'e Des\-car\-tes, 67084 Strasbourg CEDEX,
FRANCE.}
\email{robert.laterveer@math.unistra.fr}

\title{Algebraic cycles and Verra fourfolds}

\begin{abstract} This note is about the Chow ring of Verra fourfolds. For a general Verra fourfold, we show that the Chow group of homologically trivial $1$-cycles is generated by conics.
We also show that Verra fourfolds admit a multiplicative Chow--K\"unneth decomposition, and draw some consequences for the intersection product in the Chow ring of Verra fourfolds.
\end{abstract}



\maketitle

\section{Introduction}

 {\em Verra fourfolds\/} \cite{Ver} are defined as double covers of $\PP^2(\C)\times\PP^2(\C)$ branched along a smooth divisor of bidegree $(2,2)$.
 
  One could say these varieties live in a world parallel to the world of cubic fourfolds. For example, their Hodge diamond looks much like that of a K3 surface (cf. Corollary \ref{hodgediam}). Also, both cubic fourfolds and Verra fourfolds are related (via an Abel--Jacobi isomorphism) to hyperk\"ahler fourfolds: for the first, this is the famous \cite{BD}; for the second, this is the more recent \cite{IKKR}.
   
 The aim of this note is to explore the analogy 
  \[  \hbox{Verra fourfolds}\ \leftrightsquigarrow    \ \hbox{cubic\ fourfolds} \]
 on the level of Chow groups. A famous result for smooth cubic fourfolds $X$ is that the Chow group of $1$-cycles $A^3(X)$ is generated by lines \cite{Par} (cf. also \cite[Corollary 4.3]{Shen} and \cite[Theorem 1.7]{TZ} for alternative proofs). Here is the analogous statement for Verra fourfolds:
 
 \begin{nonumbering}[=Corollary \ref{conics}] 
  Let $X$ be a general Verra fourfold. The Chow group of homologically trivial $1$-cycles $A^3_{hom}(X)_{\QQ}$ is generated by $(1,1)$-conics (i.e., conics that project to lines via both projections $X\to\PP^2$).
 \end{nonumbering}
  
 This is proven using (a slight variant on) Voisin's celebrated method of {\em spread\/} of algebraic cycles in families \cite{V0}, \cite{V1}, \cite{V8}, \cite{Vo}, combined with the Abel--Jacobi type isomorphism in cohomology.
 
 Next, we show the following:
 
 \begin{nonumbering}[=Theorem \ref{verramck}] Any Verra fourfold admits a multiplicative Chow--K\"unneth decomposition, in the sense of Shen--Vial \cite{SVfourier}.
 \end{nonumbering} 
 
 For background on the concept of multiplicative Chow--K\"unneth decomposition, cf. Subsection \ref{ss:mck} below. 
  Theorem \ref{verramck} has the following consequence, reminiscent of the structure of the Chow ring of K3 surfaces \cite{BV}:
 
 \begin{nonumberingc}[=Corollary \ref{subring}] Let $X$ be any Verra fourfold. The subalgebra
   \[ R^\ast(X):= \langle A^1(X)_{\QQ}, \ A^2(X)_{\QQ}, \ c_j(T_X)\rangle\ \ \ \subset\ A^\ast(X)_{\QQ}\ \]
   injects into cohomology via the cycle class map. (Here, $c_j(T_X)$ are the Chow-theoretic Chern classes.)
   \end{nonumberingc}
   
  (This can be generalized to self-products $X^m$, cf. Corollary \ref{corm}.)
  The construction of a multiplicative Chow--K\"unneth decomposition relies once more on the spread argument, by first showing (Theorem \ref{verrak3}) that any Verra fourfold can be motivically related to a genus $2$ K3 surface. This close link to K3 surfaces also has the following consequence:
  
 \begin{nonumberingc}[=Corollary \ref{isog}] Let $X,X^\prime$ be two Verra fourfolds that are isogenous (i.e., there exists an isomorphism of $\QQ$-vector spaces $H^4(X,\QQ)\cong H^4(X^\prime,\QQ)$ compatible with the Hodge structures and with cup product). Then there is an isomorphism of Chow motives
    \[ \hh(X)\ \xrightarrow{\cong}\ \hh(X^\prime)\ \ \ \hbox{in}\ \MM_{\rm rat}\ .\]
    \end{nonumberingc}
  
 Further consequences concern a multiplicative decomposition in the derived category (Corollary \ref{deldec}), 
 the generalized Hodge conjecture for self-products, and finite-dimensionality of the motive of certain special Verra fourfolds (Corollary \ref{other}).

\vskip0.6cm

\begin{convention} In this note, the word {\sl variety\/} will refer to a reduced irreducible scheme of finite type over $\C$. For any $n$-dimensional variety $X$, we will 
write $A_i(X)$ for the Chow group of dimension $i$ cycles on $X$ with $\QQ$-coefficients, and $A^j(X)$ for the operational Chow cohomology of \cite{F} with $\QQ$-coefficients
(this can be identified with $A_{n-j}(X)$ for smooth $X$).

The notations 
$A^j_{hom}(X)$ and $A^j_{AJ}(X)$ will indicate the subgroups of 
homologically trivial (resp. Abel--Jacobi trivial) cycles.

For a morphism between smooth varieties $f\colon X\to Y$, we will write $\Gamma_f\in A^\ast(X\times Y)$ for the graph of $f$, and ${}^t \Gamma_f\in A^\ast(Y\times X)$ for the transpose correspondence.

The contravariant category of Chow motives (i.e., pure motives with respect to rational equivalence as in \cite{Sc}, \cite{MNP}) will be denoted $\MM_{\rm rat}$. 

We will write $H^\ast(X)=H^\ast(X,\QQ)$ for singular cohomology with $\QQ$-coefficients.
\end{convention}

\section{Verra fourfolds}

\begin{definition}[\cite{Ver}] A {\em Verra fourfold\/} is a double cover of $\PP^2\times\PP^2$ branched along a smooth divisor of bidegree $(2,2)$.
\end{definition}

\begin{remark} The moduli space of Verra fourfolds is $19$-dimensional \cite[Section 0.3]{IKKR}.
\end{remark}

\begin{lemma}\label{fano} A Verra fourfold $X$ is a Fano variety (i.e., the canonical bundle is anti-ample). In particular, $A^4(X)=\QQ$ and the Hodge conjecture is true for $X$.
\end{lemma}

\begin{proof} The canonical bundle formula shows that $X$ is Fano. Fano varieties are rationally connected \cite{Ca}, \cite{Kol}, and so $A^4(X)=\QQ$. The Hodge conjecture is known for fourfolds with trivial Chow group of $0$-cycles \cite{BS}.
\end{proof}

Just as for cubic fourfolds \cite{BD}, there is an Abel--Jacobi isomorphism relating the cohomology of a general Verra fourfold with that of a certain hyperk\"ahler fourfold:

\begin{theorem}[\cite{IKKR}, \cite{KKM}]\label{aj0} Let $X\to \PP^2\times\PP^2$ be a Verra fourfold. There exists a variety $Z$ (the associated double EPW quartic). Assume $Z$ is smooth (which is the case for general $X$). Then $Z$ is a hyperk\"ahler fourfold of K3$^{[2]}$-type, and
there is an isomorphism of $\QQ$-vector spaces
  \[ f_{\rm KKM}\colon\ \   H^4_{tr}(X,\QQ)\ \xrightarrow{\cong}\ H^2_{tr}(Z,\QQ)(1)\ ,\]
  compatible with cup product up to a coefficient.
   (Here $H_{tr}^\ast()$ denotes transcendental cohomology, i.e. the orthogonal complement of the subspace of algebraic classes.)
\end{theorem}

\begin{proof} One way to construct the double EPW quartic $Z$ is in terms of $(1,1)$-conics on $X$, cf. Theorem \ref{FZ} below. The isomorphism $f_{\rm KKM}$ is constructed in \cite[Section 3.1]{KKM}. The gist of the argument is to use the associated K3 surface $S$. More precisely, the projection to any of the two factors $\PP^2$ is a quadric fibration $X\to\PP^2$. The discriminant locus of any of these two fibrations is a sextic curve $C\subset\PP^2$ which has at most simple double points, and which is smooth for the generic $X$ \cite[section 1]{Beau2}. The double cover of $\PP^2$ branched over $C$ is a K3 surface $S$ with at most rational double points, and which is smooth for generic $X$. (NB: in the course of this note, we will use ``K3 surface'' to mean a surface with at most rational double points, for which the minimal desingularization is a smooth K3 surface.)

Laszlo has constructed \cite[Proposition II.3.1]{Las} an isomorphism of Hodge structures
  \begin{equation}\label{flasz} f_{\rm Lasz}\colon\ \  H^4_{prim}(X,\QQ)\ \xrightarrow{\cong}\ H^2_{prim}(S,\QQ)(1) \ ,\end{equation}
  which implies in particular that there is an isomorphism
  \[  f_{\rm Lasz}\colon\ \  H^4_{tr}(X,\QQ)\ \xrightarrow{\cong}\ H^2_{tr}(S,\QQ)(1) \ .\]
  On the other hand, it is known \cite[Proposition 4.1]{IKKR} that the double EPW quartic $Z$ associated to $X$ is isomorphic to
the moduli space of twisted stable sheaves on the $K3$ surface $S$ (such moduli spaces are constructed in \cite{Yos}):
  \[ Z\cong  M_{\nu}(S,\alpha) \ .\]
 Here $\nu$ is a Mukai vector, and $\alpha$ is a certain Brauer class. 
Building on work of Mukai, Yoshioka has constructed \cite[Theorem 3.19]{Yos} an isomorphism of Hodge structures
 \[ f_{\rm Yosh}\colon\ \ H^2(Z,\QQ)\ \xrightarrow{\cong}\  \nu^\bot\ \]
 (here the right-hand side denotes the orthogonal to $\nu$ inside the Mukai lattice $\wt{H}^\ast(S,\QQ)$). Taking the transcendental part on both sides, this induces in particular an isomorphism
  \begin{equation}\label{fyosh} f_{\rm Yosh}\colon\ \  H^2_{tr}(Z,\QQ)\ \xrightarrow{\cong}\ H^2_{tr}(S,\QQ)\ .\end{equation}
  The required isomorphism is now defined as a composition:
  \[ f_{\rm KKM}:= (f_{\rm Yosh})^{-1}\circ f_{\rm Lasz}\colon\ \  H^4_{tr}(X,\QQ)\ \xrightarrow{\cong}\ H^2_{tr}(Z,\QQ)\ .\]
  
  The compatibility with cup product follows from the fact that both $f_{\rm Lasz}$ and $f_{\rm Yosh}$ are compatible with bilinear forms up to a coefficient (cf. \cite[p. 251]{Las}, resp. \cite[Theorem 3.19]{Yos}). \end{proof}

\begin{corollary}\label{hodgediam} Let $X$ be a Verra fourfold. The Hodge diamond of $X$ is
    \[ \begin{array}[c]{ccccccccccccc}
      &&&&&& 1 &&&&&&\\
      &&&&&0&&0&&&&&\\
      &&&&0&&2&&0&&&&\\
        &&&0&&0&&0&&0&&&\\
      &&0&&1&&21&&1&&0&&\\
      &&&0&&0&&0&&0&&&\\
       &&&&0&&2&&0&&&&\\
        &&&&&0&&0&&&&&\\      
        &&&&&& 1 &&&&&&\\
\end{array}\]
 \end{corollary}

\begin{proof} Let $\bar{C}:={C}(\PP^2\times\PP^2)$ denote the cone over $\PP^2\times\PP^2$ with vertex $\nu$, and let $C^\circ:=\bar{C}\setminus \nu$ denote the open cone. The cone $\bar{C}$ lives in $\PP(\C\oplus \C^3\otimes\C^3)\cong\PP^9$.
Any Verra fourfold $X$ can be obtained by intersecting $\bar{C}$ with a quadric hypersurface: 
  \[ X=\bar{C}\cap Q\ \ \subset \PP^9\ .\]
As $C^\circ$ is an affine bundle over $\PP^2\times\PP^2$, we have
  \[ H^j(\bar{C})=H^j(C^\circ)=\begin{cases}    0 &\hbox{if}\ j\ \hbox{is\ odd}\ ,\\
                                                                        \QQ^2 &\hbox{if}\ j=2\ .\\
                                                                        \end{cases}\]
        Because $\bar{C}$ has only locally complete intersection singularities, the weak Lefschetz theorem for rational cohomology is still valid for the inclusion $X\hookrightarrow\bar{C}$
        (indeed, the affine lci variety $U:=\bar{C}\setminus X$ verifies $H^j_c(U,\QQ)=0$ for $j<5$ according to \cite[Expos\'e III Corollaire 3.11(\rom3)]{GNPP}). This gives all but the central line of the Hodge diamond.                                                                

As for the central line, $X$ is Fano (Lemma \ref{fano}) and so $h^{4,0}(X)=0$. The remaining Hodge numbers follow from the Abel--Jacobi isomorphism (Theorem \ref{aj0}): 
Indeed, for generic $Z$, the K3 surface $S$ is also generic, and hence
 \[ H^2_{tr}(Z)\cong H^4_{tr}(X)\cong H^2_{tr}(S)\] 
 is of dimension $21$. 
 For very general $X$, the space of algebraic classes inside $H^4(X)$ is $2$-dimensional (indeed, the isomorphism (\ref{flasz}), plus the fact that $H^2_{prim}(S)=H^2_{tr}(S)$ for very general degree $2$ K3 surfaces, implies that
 $H^4_{prim}(X)=H^4_{tr}(X)$ for $X$ very general). This gives $\dim H^4(X)=23$.
\end{proof}

\begin{remark} To further the analogy with cubic fourfolds, we remark that the general Verra fourfold is suspected to be irrational \cite[Remark 4.2]{IKKR}. Certain special Verra fourfolds are known to be rational \cite[Corollary 6.4]{CKKM}. No examples of irrational Verra fourfolds seem to be known.
\end{remark}

For several reasons, we need a somewhat modified version of the Abel--Jacobi isomorphism of Theorem \ref{aj0}. One of the reasons is that we want a {\em correspondence-induced\/} isomorphism.
We will rely on the following description of the double EPW quartic:

\begin{theorem}[Iliev--Kapustka--Kapustka--Ranestad \cite{IKKR}]\label{FZ} Let $X$ be a general Verra fourfold, let $Z=Z(X)$ be the associated double EPW quartic, and let $F=F_{(1,1)}(X)$ be the Hilbert scheme of $(1,1)$-conics on $X$ (i.e., conics projecting to lines under the two projections $X\to\PP^2$). Then $F$ is smooth of dimension $5$, and there is a morphism
  \[ \sigma\colon\ \ F\ \to\ Z \]
  which is a $\PP^1$-fibration.
  \end{theorem}

\begin{proof} This is \cite[Theorem 0.2]{IKKR}.
\end{proof}

Here is our home-made modified version of the Abel--Jacobi isomorphism 
of Theorem \ref{aj0}:

\begin{proposition}\label{aj} Let $X$ be a general Verra fourfold as in Theorem \ref{FZ}, let $Z$ be the associated double EPW quartic, and let $F$ be the Hilbert scheme of $(1,1)$-conics on $X$. 

\begin{enumerate}[(i)]

\item
Let
   \[ \begin{array}[c]{ccc}
       C&\xrightarrow{p}& F\\
       \ \ \  \downarrow{\scriptstyle q}&&\\
          X&&\\
          \end{array}\]
          denote the universal $(1,1)$-conic. There is an isomorphism
        \[  p_\ast q^\ast\colon\ \ H^4_{tr}(X,\QQ)\ \xrightarrow{\cong}\ \ima \Bigl( H^2_{tr}(Z,\QQ)\ \xrightarrow{\sigma^\ast}\ H^2(F,\QQ)\Bigr)\ .\]
        
   \item Assume $X$ is very general.
    The isomorphism of (\rom1) coincides with the isomorphism of Theorem \ref{aj0} up to a coefficient:
    \[ p_\ast q^\ast =  \lambda\, \sigma^\ast f_{\rm KKM}\colon\ \     H^4_{tr}(X,\QQ)\ \xrightarrow{\cong}\ \ima \Bigl( H^2_{tr}(Z,\QQ)\ \xrightarrow{\sigma^\ast}\ H^2(F,\QQ)\Bigr)\ ,\]
    for some $\lambda\in\QQ^\ast$.
    
    \item Let $\Phi$ denote the induced Abel--Jacobi isomorphism
     \[  \Phi\colon\ \ H^4_{tr}(X,\QQ)\ \xrightarrow{p_\ast q^\ast}\ H^2_{tr}(F,\QQ)\ \xrightarrow{\cdot \xi}\ H^4_{tr}(F,\QQ)\ \xrightarrow{\sigma_\ast}\ H^2_{tr}(Z,\QQ)\ ,\]
     where $\xi\in A^1(F)$ is a relatively ample class for the $\PP^1$-fibration $F\to Z$.
     
       This isomorphism respects bilinear forms up to a coefficient: for all $\alpha,\beta\in H^4_{tr}(X,\QQ)$ there is equality
    \[ \langle \alpha,\beta\rangle_X = \mu \,\langle\Phi(\alpha),  \Phi(\beta)\rangle_Z\ ,\]
    for some $\mu\in\QQ^\ast$. Here $\langle -,-\rangle_X$ denotes the cup product, and $\langle-,-\rangle_Z$ denotes the Beauville--Bogomolov form. The constant $\mu$ is the same for all Verra fourfolds.
     \end{enumerate}      
          \end{proposition}

\begin{proof} 

(\rom1) We claim that 
  \[  p_\ast q^\ast\colon\ \ H^{3,1}(X)\ \to\ H^{2,0}(F) \]
  is injective, for very general $X$. This claim suffices to prove (\rom1): indeed, the claim implies that also the composition
  \[ H^{3,1}(X)\ \xrightarrow{p_\ast q^\ast}\ H^{2,0}(F)\ \xrightarrow{\cdot \xi}\ H^{3,1}(F)\ \xrightarrow{\sigma_\ast}\ H^{2,0}(Z)\ \]
  is non-zero (hence an isomorphism), since $\sigma$ is a $\PP^1$-fibration. That is,
   \[ \Phi (H^4_{tr}(X))\  \subset\ H^2_{tr}(Z) \]
   is a non-zero Hodge substructure. However, since $\dim H^{2,0}(Z)=1$ the Hodge structure $H^2_{tr}(Z)$ is simple.
   As such, the image of $\Phi$ must be all of $H^2_{tr}(Z)$. Since both $H^4_{tr}(X)$ and $H^2_{tr}(Z)$ are $21$-dimensional, $\Phi$ is an isomorphism for very general $X$. Using a specialization argument, $\Phi$ must be an isomorphism for all $X$ for which the $\PP^1$-fibration $\sigma$ of Theorem \ref{FZ} exists.
   
Let us now establish the claim. Let $F^\prime:= F\cap h^2$ be the intersection of $F$ with $2$ general hyperplanes (so $F^\prime$ has dimension $3$), and form the fibre diagram  
    \[ \begin{array}[c]{ccc}
       C^\prime &\xrightarrow{p^\prime}& F^\prime\\
         \ \ \ \ \downarrow{\scriptstyle i_C}&&  \ \   \downarrow{\scriptstyle i}    \\
       C&\xrightarrow{p}& F\\
       \ \   \downarrow{\scriptstyle q}&&\\
          X&&\\
          \end{array}\]   
    
    To ease notation, let us write $q^\prime :=q\circ i_C\colon C^\prime\to X$. The morphism $q^\prime$ is dominant, hence generically finite. (Indeed, a general fibre $C_x$ of $q$ is $2$-dimensional, corresponding to a $2$-dimensional family of conics passing through $x\in X$. The image $p(C_x)\subset F$ is a surface; a general $F^\prime$ will meet this surface in a finite number of points). To prove the claim, we make the new claim that 
       \[ (p^\prime)_\ast (q^\prime)^\ast\colon\ \ H^{3,1}(X)\ \to\ H^{2,0}(F^\prime) \]     
      is injective, for very general $X$. To see that this new claim implies the first claim, let $\alpha\in H^{3,1}(X)$. We observe that
      \[ \begin{split}  i_\ast    (p^\prime)_\ast (q^\prime)^\ast (\alpha) &=  p_\ast (i_C)_\ast (i_C)^\ast q^\ast (\alpha)\\
                &= p_\ast \bigl(  q^\ast(\alpha) \cdot C^\prime \bigr)\\
                &= p_\ast \bigl( q^\ast(\alpha)\cdot p^\ast(h^2)\bigr) = p_\ast q^\ast(\alpha)\cdot h^2\ \ \ \hbox{in}\ H^{4,2}(F)\ .\\
                \end{split}\]
         Hard Lefschetz guarantees that $H^2(F)\cong H^2(F^\prime)\to H^6(F)$ is injective, and so the new claim indeed implies the first claim.
         
         To prove the new claim, we borrow the argument of \cite[Proposition 6]{BD}. That is, let $\EE\to F^\prime$ be the rank $2$ vector bundle such that $C^\prime$ is the projectivisation of $\EE$, and let $c_1,c_2$ denote the Chern classes of $\EE$. As $(q^\prime)^\ast(h)$ is a relative hyperplane class for the $\PP^1$-bundle $C^\prime\to F^\prime$, there is a relation
         \begin{equation}\label{star}       (q^\prime)^\ast(h^2) +(p^\prime)^\ast(c_1)\cdot (q^\prime)^\ast(h)+(p^\prime)^\ast(c_2) = 0\ \ \ \hbox{in}\   H^{4}(C^\prime)\ .\end{equation}
        
        The projective bundle formula says that for any given $\alpha\in H^{4}(X)$, one can write
            \[ (q^\prime)^\ast(\alpha)= (p^\prime)^\ast(\alpha_1)\cdot (q^\prime)^\ast(h) + (p^\prime)^\ast(\alpha_2) \ \ \ \hbox{in}\ H^4(C^\prime)\ ,\]
        where $\alpha_j\in H^{2j}(F^\prime)$. In particular, this implies that $\alpha_1=(p^\prime)_\ast(q^\prime)^\ast(\alpha)$.
        
        Multiplying with $(q^\prime)^\ast(h)$, we get
        \[ \begin{split} (q^\prime)^\ast(\alpha\cdot h)&=  (p^\prime)^\ast(\alpha_1)\cdot (q^\prime)^\ast(h^2) + (p^\prime)^\ast(\alpha_2)\cdot (q^\prime)^\ast(h)\\
                           &= (p^\prime)^\ast (\alpha_2-\alpha_1\cdot c_1)\cdot (q^\prime)^\ast(h) - (p^\prime)^\ast(\alpha_1\cdot c_2)
                           \ \ \ \ \ \ \hbox{in}\ H^4(C^\prime)\ ,\\
                           \end{split}\]       
        where in the last equality we have used relation (\ref{star}). In particular, if $\alpha$ is primitive (i.e. $\alpha\cdot h=0$), we must have
        \[ \alpha_1\cdot c_2 = \alpha_2-\alpha_1\cdot c_1=0\ .\]
        That is, for any $\alpha\in H^4_{tr}(X)$ we can write
        \[ \begin{split}  (q^\prime)^\ast(\alpha)&=  (p^\prime)^\ast (\alpha_1)\cdot (q^\prime)^\ast(h)  +(p^\prime)^\ast(\alpha_2)\\
                      &=   (p^\prime)^\ast  \Bigl(   (p^\prime)_\ast (q^\prime)^\ast (\alpha)\Bigr) \cdot (q^\prime)^\ast(h)  + (p^\prime)^\ast (\alpha_1\cdot c_1)\\
                      &= (p^\prime)^\ast  \Bigl(   (p^\prime)_\ast (q^\prime)^\ast (\alpha)\Bigr) \cdot \Bigl(  (q^\prime)^\ast(h)  + (p^\prime)^\ast(c_1)\Bigr) \ \ \ \ \ \ \hbox{in}\ H^4(C^\prime)\ .\\
                           \end{split}\]     
        
        Taking squares on both sides, we get
        \[  \begin{split}  (q^\prime)^\ast(\alpha^2)&= (p^\prime)^\ast  \Bigl(   (p^\prime)_\ast (q^\prime)^\ast (\alpha)\Bigr)^2 \cdot \Bigl(  (q^\prime)^\ast(h^2)+2(p^\prime)^\ast(c_1)\cdot (q^\prime)^\ast(h)+(p^\prime)^\ast(c_1^2) \Bigr)\\
        &= (p^\prime)^\ast  \Bigl(   (p^\prime)_\ast (q^\prime)^\ast (\alpha)\Bigr)^2 \cdot \Bigl(  (p^\prime)^\ast(c_1)\cdot (q^\prime)^\ast(h)+(p^\prime)^\ast(c_1^2-c_2)\Bigr)\\
        &=(p^\prime)^\ast  \Bigl(   (p^\prime)_\ast (q^\prime)^\ast (\alpha)\Bigr)^2 \cdot  (p^\prime)^\ast(c_1)\cdot (q^\prime)^\ast(h)        \ \ \  \ \ \ \hbox{in}\ H^8(C^\prime)\ \\
                           \end{split}\]   
                       (where the second equality uses relation (\ref{star}), and the last equality is for dimension reasons).      
           
        The left-hand side equals $d\, \alpha^2$ in $H^8(C^\prime)\cong\QQ$, where $d$ is the degree of the morphism $q^\prime$. Taking push-forward to $F^\prime$, we get
           \[         d \, \alpha^2 =     \Bigl(   (p^\prime)_\ast (q^\prime)^\ast (\alpha)\Bigr)^2 \cdot  c_1\ \ \ \hbox{in}\ H^6(F^\prime)\cong\QQ\ .\]
        This relation implies that any $\alpha,\beta\in H^4_{tr}(X)$ satisfy
        \[ (p^\prime)_\ast (q^\prime)^\ast (\alpha) \cdot  (p^\prime)_\ast (q^\prime)^\ast (\beta)  \cdot c_1 = d\, \alpha\cdot\beta \ ,\]
        and so
        \[  (p^\prime)_\ast (q^\prime)^\ast \colon\ \   H^4_{tr}(X)\ \to\ H^2(F^\prime) \]
     is injective by non-degeneracy of the cup product.   
        
    (\rom2) 
%
 Let $\Phi^{-1}$ be the inverse of the isomorphism $\Phi$ constructed in (\rom1). The composition 
   \[  \Phi^{-1}\circ f_{\rm KKM}\colon\ \ H^4_{tr}(X)\ \to\ H^4_{tr}(X) \]   
  is an isomorphism of Hodge structures. However, for very general $X$, every Hodge endomorphism of $H^4_{tr}(X)=H^4_{prim}(X)$ is homothetic to the identity (this follows from the corresponding property for very general genus $2$ K3 surfaces, in view of the isomorphism (\ref{flasz})), and so we must have
      \[ \Phi^{-1}\circ f_{\rm KKM} = \lambda\, \ide\colon\ \ H^4_{tr}(X)\ \to\ H^4_{tr}(X)\ , \]   
      for some non-zero $\lambda\in\QQ$.   
     It follows that
   \[ f_{\rm KKM} = \lambda\, \Phi\colon\ \ H^4_{tr}(X)\ \xrightarrow{\cong}\ H^2_{tr}(Z)\ .\]
   In view of the projective bundle formula, this proves (\rom2).
   
  (\rom3)
  Using a specialization argument, one reduces to very general Verra fourfolds $X$. Using (\rom2), one reduces to 
  the corresponding property for $f_{KKM}$, which is Theorem \ref{aj0}.
  \end{proof}

\begin{remark}\label{rem} Let $X$ be a very general Verra fourfold. As a consequence of Proposition \ref{aj}, we find that the map
  \[   f_{\rm KKM}\colon\ \   H^4_{tr}(X,\QQ)\ \xrightarrow{\cong}\ H^2_{tr}(Z,\QQ)(1)\ \]
  of Theorem \ref{aj0} is induced by a correspondence. Moreover, since Yoshioka's map $f_{\rm Yosh}$ of (\ref{fyosh}) is induced by a correspondence (it is defined in terms of characteristic classes of certain sheaves), we find that also Laszlo's map
  \[  f_{\rm Lasz}\colon\ \  H^4_{tr}(X,\QQ)\ \xrightarrow{\cong}\ H^2_{tr}(S,\QQ)(1) \]
  of (\ref{flasz}) is induced by a correspondence.
  Since the argument in \cite{Las} is based on (the low-degree exact sequence of) Leray spectral sequences, this is not obvious from looking directly at \cite{Las}.\footnote{Possibly, a less round-about way of constructing a correspondence doing the job would be to mimick \cite[Proposition 1]{V-1}, where such a correspondence is constructed between a general cubic fourfold containing a plane and a genus $2$ K3 surface coming from the quadric fibration. The situation for Verra fourfolds looks very similar to that of cubic fourfolds containing a plane; the only complication is that Verra fourfolds have Picard number $2$ instead of $1$, and so the argument of loc. cit. does not apply verbatim. I have tried, but failed to make this work.}
  \end{remark}

We now elaborate on Remark \ref{rem}:

\begin{proposition}\label{k3} Let $X$ be a very general Verra fourfold, and let $S$ be an associated K3 surface.

\begin{enumerate}[(i)]

\item
There exists a correspondence $\Psi\in A^3(X\times S)$ inducing an isomorphism
  \[ \Psi_\ast\colon\ \ H^4_{tr}(X)\ \xrightarrow{\cong}\ H^2_{tr}(S)_{} \ .\]
  
 \item
A very general genus $2$ K3 surface is related in this way to a Verra fourfold.

\item
This isomorphism respects bilinear forms up to a coefficient: for all $\alpha,\beta\in H^4_{tr}(X)$ there is equality
    \[ \langle \alpha,\beta\rangle_X =  \nu\,\langle \Psi_\ast(\alpha),  \Psi_\ast(\beta)\rangle_S\ ,\]
    for some $\nu\in\QQ^\ast$.
 (Here $\langle -,-\rangle_X$ and $\langle-,-\rangle_S$ denote the cup product.) The constant $\nu$ is the same for all Verra fourfolds.
\end{enumerate}

\end{proposition}

\begin{proof} As noted in the proof of Theorem \ref{aj0}, there are {\em two\/} K3 surfaces $S_1, S_2$ naturally associated to $X$, related to the two quadric fibrations $X\to\PP^2$. As $X$ is very general, the $S_j$ are smooth degree $2$ K3 surfaces.
In what follows, let $S$ be either $S_1$ or $S_2$.


(\rom1) The correspondence $\Psi$ is constructed in remark \ref{rem}.
  
(\rom2) 
As explained in \cite[9.8]{vG}, the Verra threefold (and hence also the Verra fourfold) can be reconstructed starting from $(S_j,\alpha_j)$. (Alternatively, the period map for double EPW quartics has $19$-dimensional image \cite{IKKR}. Thus, to a general point in the period domain (corresponding to a general genus $2$ K3 surface), one can associate a double EPW quartic $Z$, and to a general $Z$ one can in turn associate a Verra fourfold $X$ according to \cite[Theorem 0.2]{IKKR}.)

(\rom3) For very general $X$, the 
Hodge structure $H^4_{tr}(X)$
admits a unique polarization up to a non-zero coefficient. 

The second Hodge--Riemann bilinear relation plus the indecomposability of the Hodge structure $H^2_{tr}(X)$ imply that 
  \[   \langle\alpha,\beta\rangle_\Psi:= \langle\Psi_\ast(\alpha),\Psi_\ast(\beta)\rangle_S \]
  defines a polarization of $H^4_{tr}(X)$. Consequently, there exists a non-zero coefficient $\nu^\prime$ such that
  \[ \langle\alpha,\beta\rangle_\Psi:= \nu^\prime\, \langle\alpha),\beta\rangle_X\ . \]
Setting $\nu={1\over\nu^\prime}$, this settles (\rom3).

(Alternatively, one could deduce (\rom3) directly from Proposition \ref{aj}(\rom3), combined with the fact that $f_{\rm Yosh}$ respects bilinear forms \cite[Theorem 3.19]{Yos}.)
 \end{proof}

For later use, we elaborate some more on Proposition \ref{k3}, by adding that the correspondence $\Psi$ exists {\em universally\/}. (For details on the formalism of relative correspondences, cf. \cite[Chapter 8]{MNP}).

\begin{proposition}\label{univ} Let $\XX\to B$ denote the universal family of Verra fourfolds (cf. Notation \ref{fam}), and let $B^\prime\subset B$ be the open where $X_b$ has a smooth associated $K3$ surface. Let $\Ss\to B^\prime$ denote the universal family
of smooth degree $2$ K3 surfaces. There exists a correspondence
$\Psi\in A^{3}(\XX\times_{B^\prime} \Ss)$, with the property that for each $b\in B^\prime$, fibrewise restriction induces an isomorphism
  \[  (\Psi\vert_b)_\ast\colon\ \   H^4_{prim}(X_b)\ \xrightarrow{\cong}\ H^2_{prim}(S_b)(1) \ .\]
  This isomorphism is compatible with bilinear forms up to a non-zero coefficient.
  \end{proposition}
  
  \begin{proof} We will apply the following general principle:
  
    \begin{proposition}
    \label{spreadprinc} Let $\XX$, $\YY$ and $\ZZZ$ be families over $B$, and assume the morphisms to $B$ are smooth projective and the total spaces are smooth quasi--projective. Let
    \[   \Gamma\in\ \  A^i(\XX\times_B \ZZZ)\ 
                                 \]
 be a relative correspondence, with the property that for very general $b\in B$ there exist $\Lambda_b\in A^{\ast}(Y_b\times Z_b)$, $\Psi_b\in A^\ast(X_b\times Y_b)$  such that
             \[  \Gamma\vert_b= \Lambda_b\circ \Psi_b\ \ \hbox{in}\ H^{2i}(X_b\times Z_b)\ .\]
      Then there exist relative correspondences
      \[ \Lambda\ \ \in A^{\ast}(\YY\times_B \ZZZ)\ ,\ \ \Psi\in A^\ast(\XX\times_B \YY) \]
      with the property that for any $b\in B$
      \[ \Gamma\vert_b=(\Lambda)\vert_b \circ (\Psi)\vert_b\ \ \hbox{in}\ H^{2i}(X_b\times Z_b)\ .\]
   \end{proposition}
   
   \begin{proof} The statement is different, but this is proven by the same argument as \cite[Proposition 2.11]{moitod}, which in turn borrows
    the Hilbert schemes argument of \cite[Proposition 2.7]{V0}, \cite[Proposition 4.25]{Vo}.   
    
    By assumption, we dispose of a collection of data
   $(b,\Lambda_b, \Psi_b)$ 
    that are solutions to the splitting problem
   \[   \Gamma\vert_b= \Lambda_b\circ \Psi_b\ \ \hbox{in}\ H^{2i}(X_b\times Z_b)\ .\]
   Using Hilbert schemes, these data
   can be encoded by a countable number of algebraic varieties $p_i\colon M_i\to B$, $q_j\colon N_j\to B$, coming
   with universal objects 
     \[ \Lambda_i\subset\ \  \YY\times_{M_i}\ZZZ\ , \ \ \ \Psi_j\subset\ \ \XX\times_{N_j}\YY\ ,\]
     with the property that 
   for $m\in M_i$ and $b=p_i(m)\in B$, we have
     \[  (\Lambda_i)\vert_{m}=\Lambda_b\ \ \hbox{in}\ H^{\ast}(Y_b\times Z_b)\ .\]
     (And similarly, for $n\in N_j$ and $b=q_j(n)$, we have
     \[  (\Psi_j)\vert_{n}=\Psi_b\ \ \hbox{in}\ H^{\ast}(X_b\times Y_b) \ .)\]
     By assumption, the union of the $M_i$ dominates $B$ (and likewise, the union of the $N_j$ dominates $B$). Since there is a countable number of $M_i$, one of them (say $M_0$) must dominate $B$ (and likewise, there exists $N_0$ dominating $B$). Taking hyperplane sections, we may assume $M_0\to B$ is generically finite (say of degree $d$). Projecting $\Lambda_0$ to $\YY\times_B \ZZZ$ and dividing by $d$, we have obtained 
    a relative correspondence $\Lambda\in  A^{\ast}(\YY\times_B \ZZZ)$ as requested (and similarly, a relative correspondence $\Psi$).
   \end{proof}   
  
  To prove Proposition \ref{univ}, we apply the general principle (Proposition \ref{spreadprinc}) with
    $\XX=\ZZZ=$ the universal family of Verra fourfolds restricted to $B^\prime$, and $\YY=\Ss$, the universal family of smooth degree $2$ K3 surfaces. We take $\Gamma$ to be the ``corrected relative diagonal'' 
     \begin{equation}\label{corrdiag} \Gamma:= \Delta_\XX^-:=\Delta_\XX - \pi^0_\XX - \pi^2_\XX - \pi^6_\XX -\pi^8_\XX - \Gamma_{S_1}-\Gamma_{S_2}\ \ \ \in\ A^4(\XX\times_B \XX)
       \ ,\end{equation}
     where $\pi^i_\XX$, $i\not=4$ are defined in terms of cycles coming from (the cone over) $\PP^2\times\PP^2$, and $\Gamma_{S_j}\vert_{X_b}$, $j=1,2$ are projectors on the two surfaces $S_1, S_2$ that generically span $H^{2,2}(X_b)$ (these surfaces $S_j$ are restrictions of codimension $2$ subvarieties of the cone over $\PP^2\times\PP^2$ and so exist universally). This relative correspondence $\Gamma$ is constructed such that the fibrewise restriction $\Gamma\vert_b$ is a projector on the primitive cohomology $H^4_{prim}(X_b,\QQ)$ (which coincides with the transcendental cohomology for very general $b\in B$).
 Proposition \ref{k3} furnishes fibrewise correspondences $\Lambda_b, \Psi_b$ fulfilling the assumption of Proposition \ref{spreadprinc}. Thus, it follows from Proposition \ref{spreadprinc} that there exists a relative correspondence $\Psi\in A^3(\XX\times_B \Ss)$ as requested.   
  
  The compatibility with bilinear forms is proven as above: one restricts to a very general fibre, where it must be true by unicity.
       \end{proof}
    
  \begin{remark} We will see later (Remark \ref{later}) that the isomorphism
   \[ H^4_{prim}(X_b)\ \cong\ H^2_{prim}(S_b)(1)\] 
   of Proposition \ref{univ} can be extended to {\em all\/} Verra fourfolds, including those where the associated K3 surfaces are singular.
  \end{remark}

  \section{Various general preliminaries}
  
 \subsection{Alexander schemes} 
 \label{ss:alex}
 
 This subsection gathers some properties of singular surfaces that are Alexander schemes, in the sense of Vistoli and Kimura. These surfaces $S$ have the desirable property that the Chow groups with rational coefficients of $S$ and all its powers $S^m$ behave just as well as Chow groups of smooth varieties.

\begin{proposition}\label{alex} Let $S$ be a normal projective surface, and assume there is a resolution of singularities such that the exceptional divisor is a union of rational curves. Then $S^m$ is an Alexander scheme in the sense of \cite{Vis}, \cite{Kim0}. In particular, the formalism of correspondences with rational coefficients works for $S^m$ (and for products of $S^m$ with smooth projective varieties) just as for smooth projective varieties.
\end{proposition}

\begin{proof} The fact that $S$ is an Alexander scheme is a result of Vistoli's:

\begin{theorem}[Vistoli \cite{Vis}]\label{vis} Let $S$ be a normal surface. The following are equivalent:

\noindent
(\rom1) $S$ is an Alexander scheme;

\noindent
(\rom2) there exists a resolution of singularities of $S$ such that the exceptional divisor is a union of rational curves.
\end{theorem}

\begin{proof} This is 
\cite[Theorem 4.1]{Vis}.
\end{proof}

 Since the property ``being an Alexander scheme'' is stable under products \cite[Remark 2.7(\rom1)]{Kim0}, $S^m$ is an Alexander scheme. (NB: the fact that the definitions of Alexander scheme given in \cite{Vis} and \cite{Kim0} coincide for connected schemes is 
\cite[Corollary 4.5]{Kim0}.)

The fact that products of Alexander schemes are Alexander, and that Chow groups of an Alexander scheme have an intersection product and a pullback, means that the formalism of correspondences can be extended from smooth projective varieties to projective Alexander schemes.
\end{proof}
  
 \begin{remark}\label{alexm} Since the formalism of correspondences (with rational coefficients) extends to Alexander schemes, one may extend the category of pure Chow motives (with rational coefficients) from smooth projective varieties to projective Alexander schemes. We will only use this extension at the end of the proof of Theorem \ref{verrak3}, when dealing with Verra fourfolds whose associated K3 surfaces are singular.
 \end{remark}

\subsection{Transcendental part of the motive}

\begin{theorem}[Kahn--Murre--Pedrini \cite{KMP}]\label{t2} Let $S$ be a surface. There exists a decomposition
  \[ \hh^2(S)= \ttt^2(S)\oplus \hh^2_{alg}(S)\ \ \ \hbox{in}\ \MM_{\rm rat}\ ,\]
  such that
  \[  H^\ast(\ttt^2(S),\QQ)= H^2_{tr}(S)\ ,\ \ H^\ast(\hh^2_{alg}(S),\QQ)=NS(S)_{\QQ}\ \]
  (here $H^2_{tr}(S)$ is defined as the orthogonal complement of the N\'eron--Severi group $NS(S)_{\QQ}$ in $H^2(S,\QQ)$),
  and
   \[ A^\ast(\ttt^2(S))_{}=A^2_{AJ}(S)_{}\ .\]
   (The motive $\ttt^2(S)$ is called the {\em transcendental part of the motive\/}.)
   \end{theorem} 

\begin{proposition} Let $X$ be a Verra fourfold. There exists a decomposition
  \[ \hh(X)= \ttt^4(X)\oplus \hh^4_{alg}(X)\oplus \bigoplus_{0\le j\le 4, j\not=2} \mathds{1}(j)\ \ \ \hbox{in}\ \MM_{\rm rat}\ , \]
such that
    \[  H^\ast(\ttt^4(X),\QQ)= H^4_{tr}(X)\ ,\ \ H^\ast(\hh^4_{alg}(X),\QQ)=N^{ 2}(X):=\ima\bigl( A^{ 2}(X)\to H^4(X,\QQ)\bigr)\ ,\]
  and
   \[ A^\ast(\ttt^4(X))_{}=A^\ast_{hom}(X)_{} = A^{3}_{hom}(X)  \ .\]
   (The motive $\ttt^4(X)$ will be called the {\em transcendental part of the motive\/}.)
\end{proposition}

\begin{proof} This is a standard construction (cf. \cite[Section 4]{Pedr}, where this decomposition is constructed for cubic fourfolds. Cf. also \cite[Theorem 1]{V5}, where projectors on the algebraic part of cohomology are constructed for any smooth projective variety satisfying the standard conjectures).

\end{proof}

\subsection{Chow--K\"unneth property}

\begin{definition}[Totaro \cite{T}] A quasi-projective variety $M$ has the {\em Chow--K\"unneth property\/} if for any quasi--projective variety $Z$, the natural map
  \[  A_\ast(M)\otimes A_\ast(Z)\ \to\ A_\ast(M\times Z) \]
  is an isomorphism.
  \end{definition}

\begin{proposition}[Totaro \cite{T}]\label{propck}

\noindent
\item{
(\rom1)} Let $M$ be a quasi-projective variety that is a {\em linear variety\/} in the sense of \cite{T} (in particular, $M$ may be a Grassmannian, a toric or spherical variety). Then $M$ has the Chow--K\"unneth property;

\noindent
\item{
(\rom2)} Let $M\to N$ be an affine or projective bundle. If $N$ has the Chow--K\"unneth property, then also $M$ has the Chow--K\"unneth property;

\noindent
\item{(\rom3)} Let $M$ be a quasi-projective variety with the Chow--K\"unneth property. The cycle class maps induce isomorphisms
 \[                        A_i(M)\   \xrightarrow{\cong}\ W_{-2i} H^{BM}_{2i}(M,\QQ)\ , 
                        \]
                       where $W_\ast$ denotes Deligne's weight filtration.
\end{proposition}

\begin{proof} Point (\rom1) is \cite[Proposition 1]{T}. Point (\rom2) follows readily from the affine resp. projective bundle formula. Point (\rom3) is
\cite[Theorem 3]{T}.
\end{proof}

\begin{remark} Property (\rom3) shows that the Chow--K\"unneth property implies ``having trivial Chow groups'', a notion studied by Voisin \cite{V0}, \cite{V1}. 
While the notion of ``having trivial Chow groups'' is well-behaved for smooth projective varieties, it is more problematic for open or singular varieties (e.g., it is not clear whether ``having trivial Chow groups'' is closed under taking products of varieties). The Chow--K\"unneth property avoids some of these issues (clearly, it is closed under taking products of varieties).
\end{remark}

\subsection{Voisin's spread, revisited}

This subsection contains a variant of a result of Voisin \cite{V1}, concerning relative correspondences for the universal family of complete intersections of a certain type.  The novelty 
consists in allowing the ambient space to have isolated singular points; this notably applies to Verra fourfolds where the ambient space is a cone.

\begin{proposition}[Voisin \cite{V1}]\label{spread} Let $\bar{M}$ be a projective variety of dimension $n+r$ with finite-dimensional singular locus $T$, and let $M:=\bar{M}\setminus T$. 
Given very ample line bundles $L_1,\ldots,L_r$ on $\bar{M}$, let 
  \[ \XX\ \to\ B \]
  denote the universal family of $n$-dimensional smooth complete intersections of type $(L_1,\ldots,L_r)$ on $\bar{M}$ avoiding $T$. Assume that for very general $b\in B$,
  the fibre $X_b$ has non-zero primitive cohomology (i.e. $H^n(M)\to H^n(X_b)$ is not surjective).

Let
  \[ R\in\ \ A^n(\XX\times_B \XX) \]
  be a cycle such that the restriction
  \[  R\vert_{X_b\times X_b}\ \ \in\ A^n(X_b\times X_b) \]
  is homologically trivial, for very general $b\in B$. Then there exists
  \[ \gamma\in\ \   A^n(M\times M) \]
  with the property that
   \[   \bigl( R+\gamma\bigr)\vert_{X_b\times X_b}=0\ \ \in\ A^n(X_b\times X_b)\ \ \ \forall b\in B\ . \]
\end{proposition}

\begin{proof} This is (essentially) the argument of \cite[Proposition 1.6]{V1} (cf. also \cite[Proposition 5.1]{LNP} for a similar statement). Since the assumptions are different, we provide a stand-alone proof.

We consider the blow-up $\wt{{M}\times{M}}$ of ${M}\times{M}$ along the diagonal, and the quotient morphism $\mu\colon\wt{{M}\times{M}}\to M^{[2]}$ to the Hilbert scheme of length $2$ subschemes. Let $\bar{B}:=\PP H^0(\bar{M},\oplus_j L_j)$
and as in \cite[Lemma 1.3]{V1}, introduce the incidence variety
      \[ I:=\bigl\{ (\sigma,y)\in \bar{B}\times\wt{{M}\times{M}}\ \vert\ s\vert_{\mu(y)}=0\bigr\}\ .\]
Since the $L_j$ are very ample on $\bar{M}$, the variety $I$ has the structure of a projective bundle over $\wt{M\times M}$. 

Next, let us consider 
  \[ p\colon\ \ \wt{\XX\times_B \XX}\ \to\ \XX\times_B \XX\ ,\]
  the blow-up along the relative diagonal $\Delta_\XX$. There is an open inclusion $\wt{\XX\times_B \XX}\subset I$.
  Hence, given $R\in A^n(\XX\times_B \XX)$ as in the proposition, there exists a (non-canonical) cycle
  $\bar{R}\in A^n(I)$ such that
  \[  \bar{R}\vert_{\wt{\XX\times_B \XX}}= p^\ast(R)\ \ \ \hbox{in}\ A^n({\wt{\XX\times_B \XX}})\ .\]
  Hence, we have
  \[ \bar{R}\vert_{\wt{X_b\times X_b}}= \bigl(p^\ast(R)\bigr)\vert_{\wt{X_b\times X_b}}= (p_b)^\ast(R\vert_{\wt{X_b\times X_b}})    =0\ \ \ \hbox{in}\ H^{2n}(    \wt{X_b\times X_b})\ ,\]
  for $b\in B$ very general, by assumption on $R$. (Here, as one might guess, $p_b\colon \wt{X_b\times X_b}\to X_b\times X_b$ denotes the blow-up along the diagonal $\Delta_{X_b}$.)
  
  Adhering to the notation of \cite[Proof of proposition 1.6]{V1}, let $\delta\in A^1(I)$ be the restriction of the exceptional divisor of $\wt{M\times M}$, and let $h\in A^1(I)$ denote the tautological class with respect to the projective bundle structure. Combining the blow-up formula and the projective bundle formula, the cycle $\bar{R}\in A^n(I)$ can be written
  \[ \bar{R}=\sum_{r,s} h^r \delta^s \pi^\ast(\gamma_{r,s})\ ,\ \ \ \gamma_{r,s}\ \ \in A^{n-r-s}(M\times M) \ ,\]
  where $\pi\colon I\to M\times M$ denotes the composition of projection and blow-up morphism.
  Let us look at the restriction of $\bar{R}$ to the fibres. On the one hand, by assumption on $R$, for $b\in B$ very general we have that
   \begin{equation}\label{assu}  (p_b)_\ast (\bar{R}\vert_{\wt{X_b\times X_b}})=0\ \ \ \hbox{in}\ H^{2n}(X_b\times X_b)\ .\end{equation}
   On the other hand, since (for all $b\in B$)
    \[ (p_b)_\ast (\delta_b^k)=\begin{cases} 0\ \ \ \hbox{in}\ A^n(X_b\times X_b)\ ,&\hbox{if}\ 0<k<n\ ,\\
                (-1)^{n-1}\Delta_{X_b}\ \ \ \hbox{in}\ A^n(X_b\times X_b)\ ,&\hbox{if}\ k=n\ ,\\
                \end{cases}\]
          we find that
        \[ (p_b)_\ast (\bar{R}\vert_{\wt{X_b\times X_b}})= \pm\gamma_{0,n}\Delta_{X_b}+ \gamma_{0,0}\vert_{{X_b\times X_b}}\ \ \ \hbox{in}\ A^n(X_b\times X_b)\ \ \ \forall b\in B\ , \]
        where $\gamma_{0,n}\in A^0(M\times M)\cong\QQ$ is some rational number.
        
      Equality (\ref{assu}) forces $\gamma_{0,n}$ to be zero. (Indeed, supposing $\gamma_{0,n}\not=0$, one would have a homological equivalence
      \begin{equation}\label{hom} \Delta_{X_b}=\pm{1\over \gamma_{0,n}} \,   \gamma_{0,0}\vert_{{X_b\times X_b}}  \ \ \ \hbox{in}\ H^{2n}(X_b\times X_b)\ ,\end{equation}
      for very general $b\in B$.
      Taking $N\supset M$ a smooth compactification of $M$, one can write $\gamma_{0,0}\vert_{X_b\times X_b}  =\bar{\gamma_{0,0}}\vert_{{X_b\times X_b}}$ for some
      $\bar{\gamma_{0,0}}\in A^n(N\times N)$. But then, letting both sides of (\ref{hom}) act on $H^n(X_b)$, one would have that $H^n(X_b)$ comes from $H^n(N)$, hence from $H^n(M)$, contradicting the assumption on primitive cohomology.)
    \end{proof}

\subsection{Multiplicative Chow--K\"unneth decomposition}
\label{ss:mck}

	\begin{definition}[Murre \cite{Mur}]\label{ck} Let $X$ be a smooth projective
		variety of dimension $n$. We say that $X$ has a 
		{\em CK  decomposition\/} if there exists a decomposition of the
		diagonal
		\[ \Delta_X= \pi^0_X+ \pi^1_X+\cdots +\pi^{2n}_X\ \ \ \hbox{in}\
		A^n(X\times X)\ ,\]
		such that the $\pi^i_X$ are mutually orthogonal idempotents and
		$(\pi^i_X)_\ast H^\ast(X)= H^i(X)$.
		Given a CK decomposition for $X$, we set 
		$$A^i(X)_{(j)} := (\pi_X^{2i-j})_\ast A^i(X).$$
		The CK decomposition is said to be {\em self-dual\/} if
		\[ \pi^i_X = {}^t \pi^{2n-i}_X\ \ \ \hbox{in}\ A^n(X\times X)\ \ \ \forall
		i\ .\]
		(Here ${}^t \pi$ denotes the transpose of a cycle $\pi$.)
		
		  (NB: ``CK decomposition'' is short-hand for ``Chow--K\"unneth
		decomposition''.)
	\end{definition}
	
	\begin{remark} \label{R:Murre} The existence of a Chow--K\"unneth decomposition
		for any smooth projective variety is part of Murre's conjectures \cite{Mur},
		\cite{MNP}. 
		It is expected that for any $X$ with a CK
		decomposition, one has
		\begin{equation*}\label{hope} A^i(X)_{(j)}\stackrel{??}{=}0\ \ \ \hbox{for}\
		j<0\ ,\ \ \ A^i(X)_{(0)}\cap A^i_{num}(X)\stackrel{??}{=}0.
		\end{equation*}
		These are Murre's conjectures B and D, respectively.
	\end{remark}

	\begin{definition}[Definition 8.1 in \cite{SVfourier}]\label{mck} Let $X$ be a
		smooth
		projective variety of dimension $n$. Let $\Delta_X^{sm}\in A^{2n}(X\times
		X\times X)$ be the class of the small diagonal
		\[ \Delta_X^{sm}:=\bigl\{ (x,x,x) : x\in X\bigr\}\ \subset\ X\times
		X\times X\ .\]
		A CK decomposition $\{\pi^i_X\}$ of $X$ is {\em multiplicative\/}
		if it satisfies
		\[ \pi^k_X\circ \Delta_X^{sm}\circ (\pi^i_X\otimes \pi^j_X)=0\ \ \ \hbox{in}\
		A^{2n}(X\times X\times X)\ \ \ \hbox{for\ all\ }i+j\not=k\ .\]
		In that case,
		\[ A^i(X)_{(j)}:= (\pi_X^{2i-j})_\ast A^i(X)\]
		defines a bigraded ring structure on the Chow ring\,; that is, the
		intersection product has the property that 
		\[  \ima \Bigl(A^i(X)_{(j)}\otimes A^{i^\prime}(X)_{(j^\prime)}
		\xrightarrow{\cdot} A^{i+i^\prime}(X)\Bigr)\ \subseteq\ 
		A^{i+i^\prime}(X)_{(j+j^\prime)}\ .\]
		
		(For brevity, we will write {\em MCK decomposition\/} for ``multiplicative Chow--K\"unneth decomposition''.)
	\end{definition}

\begin{remark}	The property of having an MCK decomposition is
	severely restrictive, and is closely related to Beauville's ``(weak) splitting
	property'' \cite{Beau3}. For more ample discussion, and examples of varieties
	admitting an MCK decomposition, we refer to
	\cite[Chapter 8]{SVfourier}, as well as \cite{V6}, \cite{SV},
	\cite{FTV}, \cite[Sections 5 and 6]{FV}, \cite{LV}.
	\end{remark}

	\begin{proposition}[Shen--Vial \cite{SV}]\label{product} Let $M,N$ be smooth projective varieties that have an MCK decomposition. Then the product $M\times N$ has an MCK decomposition.
\end{proposition}

\begin{proof} This is \cite[Theorem 8.6]{SV}, which shows more precisely that the {\em product CK decomposition\/}
  \[ \pi^i_{M\times N}:= \sum_{k+\ell=i} \pi^k_M\times \pi^\ell_N\ \ \ \in A^{\dim M+\dim N}\bigl((M\times N)\times (M\times N)\bigr) \]
  is multiplicative.
\end{proof}

\section{Main results} 


\subsection{Generalized Bloch conjecture}
The key result of this note is a ``generalized Bloch conjecture'' type of statement for universally defined correspondences (for the generalized Bloch conjecture, cf. \cite{B} and \cite{Vo}). The results announced in the introduction (Corollaries \ref{conics} and \ref{subring} and Theorem \ref{verramck}) will be deduced from Theorem \ref{main}.

\begin{notation}\label{fam} Let $\bar{C}:=C(\PP^2\times\PP^2)$ be the cone over $\PP^2\times\PP^2$, and let
  \[ B\subset \bar{B}:=\PP H^0(\bar{C},  \OO(2)) \]
 be the open parametrizing smooth dimensionally transverse intersections of $\bar{C}\subset \PP^9$ with a quadric (i.e., $B$ parametrizes all Verra fourfolds).
  We will write
   \[ \XX\ \to\ B \]
  for the universal family of Verra fourfolds, i.e.
   \[ \XX:=\bigl\{ (x,\sigma)\in \operatorname{C}(\PP^2\times \PP^2)\times B\ \big\vert\  \  \sigma\vert_x=0\bigr\}\ .\]
    \end{notation}

\begin{remark} By assumption, the fibres $X_b$ over $b\in B$ are smooth, i.e. they avoid the vertex $\nu$ of the cone $\bar{C}$. As such, the family $\XX$ enters into the set-up of Proposition \ref{spread}.
\end{remark}

\subsection{Generalized Bloch conjecture}

\begin{theorem}\label{main} Let $\XX\to B$ denote the universal family of Verra fourfolds (cf. Notation \ref{fam}). Let
$\Gamma\in A^4(\XX\times_B \XX)$ be a relative correspondence such that
  \[ (\Gamma\vert_{X_b\times X_b})_\ast H^{3,1}(X_b)=0\ \ \ \ \hbox{for\ very\ general\ }b\in B\ .\]
  Then also
  \[  (\Gamma\vert_{X_b\times X_b})_\ast A^\ast_{hom}(X_b)=0\ \ \ \forall b\in B\ .\]
\end{theorem}

 \begin{proof} 
 
 The assumption on $\Gamma$ (plus the shape of the Hodge diamond of $X_b$, and the fact that the Hodge conjecture is true for $X_b$ by Lemma \ref{fano}) implies that for the very general $b\in B$ there exist $2$-dimensional subvarieties
 $V_b^i,W_b^i$ such that
   \[ (\Gamma\circ\Delta_\XX^-)\vert_{X_b\times X_b}= \sum_{i=1}^s V_b^i\times W_b^i\ \ \ \hbox{in}\ H^{8}(X_b\times X_b)\ .\]
  (Here $\Delta_\XX^-$ is the ``corrected relative diagonal'' as in (\ref{corrdiag}).)
  By Noether--Lefschetz, the subvarieties  $V_b^i,W_b^i$ are actually obtained by restriction from subvarieties of (the cone over) $\PP^2\times\PP^2$, hence they exist universally. (Instead of evoking Noether--Lefschetz, one could also apply Voisin's Hilbert scheme argument \cite[Proposition 3.7]{V0} to obtain that the $V_b^i,W_b^i$ exist universally). That is, there exist codimension $2$ subvarieties $\VV^i, \WW^i\subset\ \XX$ such that
  \[ \bigl(  \Gamma\circ\Delta_\XX^- -  \sum_{i=1}^s \VV^i\times_B \WW^i\bigr)\vert_{X_b\times X_b}=0 \ \ \ \hbox{in}\ H^{8}(X_b\times X_b)\ ,\ \ \ \hbox{for\ very\ general\ }b\in B\ .\] 
 
 We now wave the magic wand of Proposition \ref{spread} over the relative correspondence 
   \[ R:= \Gamma\circ\Delta_\XX^- - \sum_{i=1}^s \VV^i\times_B \WW^i\ \ \in A^4(\XX\times_B \XX)\ .\] 
   The effect is that there is a cycle
 $\gamma\in A^4(M\times M)$ (where $M$ is the cone over $\PP^2\times\PP^2$ minus its vertex) with the property that
    \begin{equation}\label{prope} \bigl(  \Gamma - \sum_{j\not=4}\Gamma\circ\pi^j_\XX- \sum_{i=1}^s \VV^i\times_B \WW^i+\gamma\bigr)\vert_{X_b\times X_b}=0 \ \ \ \hbox{in}\ A^{4}(X_b\times X_b) \ \ \ \forall b\in B\ .\end{equation} 
   The quasi-projective variety $M$, being an affine bundle over a cellular variety, has the Chow--K\"unneth property (Proposition \ref{propck}). In particular, the cycle $\gamma$ decomposes as
   \[ \gamma=\sum_{j=1}^t \alpha^j\times \beta^j\ \ \ \hbox{in}\ A^4(M\times M)\ \]
   (where $\hbox{codim}(\alpha^j)+\hbox{codim}(\beta^j)=4$ for each $j$). For any given $b\in B$, one can find representatives for $\alpha^j,\beta^j, \VV^i, \WW^i$ that lie in general position with respect to $X_b$. The same applies to $\pi^j_\XX\vert_{X_b}$ with $j\not=4$, which also (by construction) comes from $M\times M$.
    Substituting this in (\ref{prope}), the upshot is that for any $b\in B$, there is a rational equivalence
    \[ \Gamma\vert_{X_b\times X_b}=\sum_{i=1}^{s+t} V^i_b\times W^i_b\ \ \ \hbox{in}\ A^4(X_b\times X_b)\ ,\]
    where $V_b^i,W_b^i\subset X_b$ are subvarieties and $\dim (V_b^i)+\dim (W_b^i)=4$ for all $i$, i.e., $    \Gamma\vert_{X_b\times X_b}$ is a {\em completely decomposed cycle}.
    As is well-known (\cite{BS}, \cite{moi}), completely decomposed cycles do not act on $A^\ast_{hom}(X_b)$, whence the conclusion.
     \end{proof}

 \subsection{Conics and $1$-cycles}
 The results stated in the introduction can be obtained as special cases of Theorem \ref{main}:
 
 \begin{corollary}\label{conics} Let $X$ be a Verra fourfold, let $F:=F_{(1,1)}(X)$ denote the Hilbert scheme of $(1,1)$-conics, and let 
   \[ \begin{array}[c]{ccc}
       C&\xrightarrow{p}& F\\
       \ \ \ \ \downarrow{\scriptstyle q}&&\\
          X&&\\
          \end{array}\]
          be the universal $(1,1)$-conic. Assume that $F$ is non-singular.
          Then there exists a correspondence $\Psi\in A^{5}(X\times F)$ such that
          \[  A^3_{hom}(X)\ \xrightarrow{\Psi_\ast}\ A^5(F)\ \xrightarrow{q_\ast p^\ast}\ A^3(X) \]
          is the identity.
             \end{corollary}

 \begin{proof} Let $Z$ denote the double EPW quartic associated to $X$. Taking $X$ sufficiently general, we may assume $Z$ is smooth.   
  Proposition \ref{aj} states that there is a correspondence $\Omega$ such that
    \begin{equation}\label{ome}  \Omega_\ast p_\ast q^\ast\colon\ \    H^4(X)_{0}\ \to\   H^2(Z)_0(1) \end{equation}
   is an isomorphism (precisely, $\Omega$ consists of intersecting with a relatively ample class for the fibration $\sigma\colon F\to Z$, followed by projection to $Z$). 
   As noted before, both $X$ and $Z$ verify the standard conjectures, and so their K\"unneth components are algebraic. The isomorphism (\ref{ome}), plus Manin's identity principle, implies that there is an isomorphism of homological motives
   \[ \Omega\circ \Gamma_p\circ {}^t \Gamma_q\colon\ \ h^4(X)\oplus {\mathds{1}}(2)^{\oplus 2}\ \xrightarrow{\cong}\ h^2(Z)(1)\oplus {\mathds{1}}(1)\ \ \ \hbox{in}\ 
       \MM_{\rm hom}\ .\]
  This means there exists a correspondence $\Phi$ which is inverse to $\Omega\circ \Gamma_p\circ {}^t \Gamma_q$, and so in particular
  there is equality
  \[ (\Phi\circ  \Omega\circ \Gamma_p\circ {}^t \Gamma_q)_\ast=\ide\colon\ \  H^4_{tr}(X,\QQ)\ \to\ H^4_{tr}(X,\QQ)\ \]
  (here $H^4_{tr}(X,\QQ)$ denotes the orthogonal complement of the algebraic part with respect to the polarization).
  Taking transpose on both sides, we obtain that also
  \[ (\Gamma_q\circ {}^t \Gamma_p\circ {}^t \Omega\circ {}^t \Phi)_\ast=\ide\colon\ \ H^4_{tr}(X,\QQ)\ \to\ H^4_{tr}(X,\QQ)\ ,\]
  (and so in particular
  $ q_\ast p^\ast\colon H^2_{tr}(F,\QQ)\to H^4_{tr}(X,\QQ)$
  is surjective).
   For brevity, let us write
   \[ \Psi:=  {}^t \Omega\circ {}^t \Phi\ \ \in A^{6}(X\times F)\ ,\]
   so that according to the above we have that
    \[  H^4_{tr}(X)\ \xrightarrow{\Psi_\ast}\ H^{8}_{tr}(F)\ \xrightarrow{q_\ast p^\ast}\ H^4_{tr}(X) \]
    is the identity.
    
    Now, we can also consider the above construction family-wise. As before, let $\XX\to B$ denote the universal family of GM fourfolds, and let $\FF\to B$ denote the universal family of Hilbert schemes of $(1,1)$-conics. There is a diagram of $B$-schemes
    \[      \begin{array}[c]{ccc}
       \CC&\xrightarrow{p}& \FF\\
       \ \ \ \ \downarrow{\scriptstyle q}&&\\
          \XX&&\\
          \end{array}\]
  (where $\CC$ is the ``relative universal $(1,1)$-conic''). 
  Let $B^\prime\subset B$ denote the Zariski open where the Hilbert schemes $F_b$ are non-singular,
  and let $B^{\prime\prime}\subset B^\prime$ denote the Zariski open where both $F_b$ and the double EPW quartic $Z_b$ are non-singular. By the above construction, for each $b\in B^{\prime\prime}$ there exists a correspondence $\Psi_b\in A^{6}(X_b\times F_b)$ such that
  \[ \Bigl( \bigl( \Gamma_q\circ {}^t \Gamma_p\bigr)\vert_{F_b\times X_b}\circ \Psi_b\Bigr){}_\ast=\ide\colon\ \ H^4_{tr}(X_b)\ \to\ H^4_{tr}(X_b)\ .\]
  Using Voisin's Hilbert schemes argument \cite[Proposition 3.7]{V0} (cf. also \cite[Proposition 2.11]{moitod} for a very similar statement), these fibrewise cycles $\Psi_b$ can be spread out. More precisely, there exists a relative correspondence
  $\Psi\in A^6(\XX\times_B \FF)$ doing the job of the various $\Psi_b$, i.e. 
  \[  \Bigl(\bigl( \Gamma_q\circ {}^t \Gamma_p\circ \Psi\bigr)\vert_{F_b\times X_b}\Bigr){}_\ast=\ide\colon\ \ H^4_{tr}(X_b)\ \to\ H^4_{tr}(X_b)\ \ \ \forall b\in B^{\prime\prime}\ .\]     

Theorem \ref{main}, applied to the relative correspondence
  \[ \Gamma:=   \Gamma_q\circ {}^t \Gamma_p\circ \Psi-\Delta_\XX\ \ \in A^4(\XX\times_{B^\prime} \XX) \]
  gives us that 
    \[ \Bigl(\bigl(  \Gamma_q\circ {}^t \Gamma_p\circ \Psi\bigr)\vert_{X_b\times X_b}\Bigr){}_\ast=\ide\colon\ \ A^3_{hom}(X_b)\ \to\ A^3_{hom}(X_b)\ \ \ \forall b\in B^\prime\ .\]
In particular, this means that there is a surjection
  \[  (q_b)_\ast (p_b)^\ast\colon\ \   A^5_{hom}(F_b)\ \twoheadrightarrow\ A^3_{hom}(X_b)  \ \ \ \forall b\in B^\prime\ .\]
Corollary \ref{conics} is now proven.
   \end{proof}

 \begin{remark} Regarding Corollary \ref{conics}, it is worth pointing out that for any rationally connected variety $X$ (in particular, for a Verra fourfold), one knows that the group of $1$-cycles $A_1(X)_{\Z}$ is generated by rational curves \cite[Theorem 1.3]{TZ}.
  \end{remark}

  \begin{remark} It would be interesting to prove that Corollary \ref{conics} is true for Chow groups with {\em integer coefficients\/}. Unfortunately, the method employed here seems ill-adapted for proving statements with integer coefficients.
  
   Also, it would be more satisfactory to obtain a statement for {\em all\/} smooth Verra fourfolds. Unfortunately, when the Hilbert scheme of $(1,1)$-conics is singular, the correspondence-type argument employed here becomes problematic.
       \end{remark}

  \subsection{A motivic relation}

 Before establishing the existence of an MCK decomposition, we first prove a statement that may be of independent interest: a motivic relation between a Verra fourfold and its associated K3 surface. Let us make precise what we mean by ``associated K3 surface'':
 
 \begin{definition}\label{assk3} Let $X$ be a Verra fourfold, and let $\bar{S}_1, \bar{S}_2$ be the double cover of $\PP^2$ branched along the discriminant loci of the two quadric fibrations $X\to\PP^2$. An {\em associated K3 surface\/} is a minimal resolution
  \[ S_j\ \to\ \bar{S}_j\ ,\]
  where $j=1,2$.
  \end{definition}
  
  With this definition, every Verra fourfold has $2$ associated K3 surfaces.
      
 \begin{theorem}\label{verrak3} Let $X$ be a Verra fourfold, and let $S$ be an associated K3 surface (cf. Definition \ref{assk3}). There is
an isomorphism of Chow motives
  \[ \Phi\colon\ \ \ttt^4(X)\ \xrightarrow{\cong}\ \ttt^2(S)(1)\ \ \ \hbox{in}\ \MM_{\rm rat}\ .\]
   \end{theorem}  
   
  \begin{proof} As before, let $\XX\to B$ denote the universal family of Verra fourfolds, and $\Ss\to B$ the universal family of degree $2$ K3 surfaces.
  As we have seen (Propositions \ref{k3} and \ref{univ}), there is a smaller non-empty set $ B^\prime\subset B$, and a correspondence
    \[ \Phi\ \ \in\ A^3(\XX\times_{B^\prime} \Ss)\ ,\]
    with the property that
    \[  (\Phi_b)_\ast\colon\ \ H^4_{tr}(X_b)\ \xrightarrow{\cong}\ H^2_{tr}(S_b)\ \]
                   is an isomorphism, for all $b\in B^{\prime}$. 
  
  We state a lemma:
                   
    \begin{lemma} Set-up as above. There exists a relative correspondence 
      \[ \Phi^{-1}\ \ \in\ A^3(\Ss\times_{B^\prime} \XX)    \ ,\]
    which fibrewise induces the inverse isomorphism: 
      \[  (\Phi^{-1}\vert_b)_\ast (\Phi\vert_b)_\ast =\ide\colon\ \ H^4_{tr}(X_b)\ \to\ H^4_{tr}(X_b)\ \ \ \forall b\in B^\prime\ .\]    
      \end{lemma}               
         
       \begin{proof} This is a standard argument. The isomorphism of cohomology groups translates into an isomorphism of cohomological motives
        \[ \Phi_b\colon\ \ \ttt^4(X_b)\ \xrightarrow{\cong}\ \ttt^2(S_b)(1)\ \ \ \hbox{in}\ \MM_{\rm hom}\ ,\]       
        for any $b\in B^\prime$. Thus, for any $b\in B^\prime$ there exists an inverse correspondence $\Phi_b^{-1}\in A^3(S_b\times X_b)$. An application of proposition \ref{spreadprinc}
        then gives a relative correspondence $\Phi^{-1}$ doing the job.
        \end{proof}

    Let us now consider the relative correspondence
    \begin{equation}\label{Gamma} \Gamma:= \Phi^{-1}\circ \Delta^-_\Ss\circ \Phi-\Delta^-_\XX\ \ \in\ A^4(\XX\times_{B^\prime}\XX)\ ,\end{equation}
    where $\Delta_\XX^-$ is the corrected relative diagonal as in (\ref{corrdiag}), and $\Delta^-_\Ss$ is the corrected relative diagonal for the universal family of degree $2$ K3 surfaces defined similarly (i.e., by subtracting $\pi^0, \pi^4$ and the relative cycle restricting to $h^2\times h^2\subset S_b\times S_b$, where $h^2$ is a general complete intersection).  
    This has the property that
    \[ (\Gamma_b)_\ast =0\colon\ \ H^4_{tr}(X_b)\ \to\ H^4_{tr}(X_b)\ \ \ \forall b\in B^{\prime}\ .\]
    Applying Theorem \ref{main} to $\Gamma$, it follows that
    \[ (\Phi^{-1}\vert_b)_\ast (\Phi\vert_b)_\ast=\ide\colon\ \ A^\ast_{hom}(X_b)\ \to\ A^\ast_{hom}(X_b)\ \ \ \forall b\in B^\prime\ ,\]
%
   and so we get injections
   \[  (\Phi_b)_\ast\colon\ \ A^3_{hom}(X_b)\ \hookrightarrow\ A^2_{hom}(S_b) \]
   for all $b\in B^\prime$. 
   
   We now proceed to upgrade from Chow groups to Chow motives. Inspecting the proof of Theorem \ref{main}, we observe that Theorem \ref{main} actually gives a bit more: given a relative correspondence $\Gamma\in A^4(\XX\times_{B^\prime}\XX)$ which is fibrewise homologically trivial, there exists a relative correspondence $\delta$ such that on the one hand
   \[ (\Gamma-\delta)\vert_b=0\ \ \ \hbox{in}\ A^4(X_b\times X_b)  \]
   for all $b\in B^\prime$, and on the other hand the restriction $\delta\vert_b$ may be assumed to be completely decomposed, for any $b\in B^\prime$. That is, in the above argument we actually have an equality of correspondences
   \begin{equation}\label{eq} (\Phi^{-1}\vert_b)\circ (\Delta_{S_b}^-)\circ (\Phi\vert_b) = \Delta_{X_b}^- +\delta\vert_b\ \ \ \hbox{in}\ A^4(X_b\times X_b)  \ \ \ \forall b\in B\ ,\end{equation}
   with $\delta\vert_b$ completely decomposed.
  The correspondence $\delta\vert_b$, being completely decomposed, has the property that
   \begin{equation}\label{theprop} \pi^{4,tr}_{X_b}\circ (\delta\vert_b)\circ   \pi^{4,tr}_{X_b}=0\ \ \ \hbox{in}\ A^4(X_b\times X_b)\ ,\end{equation} 
   where $\pi^{4,tr}_{X_b}$ is the projector defining the transcendental motive $\ttt^4(X_b)$, which is a submotive of the motive $(X_b,\Delta^-_{X_b},0)$. (To prove property (\ref{theprop}), 
 one reasons just as in \cite[Theorem 7.4.3]{KMP}, where this is proven for surfaces.)
  
  Hence, taking equality (\ref{eq}) and composing on both sides with $\pi^{4,tr}_{X_b}$, we obtain
    \[\pi^{4,tr}_{X_b}\circ (\Phi^{-1}\vert_b)\circ (\Delta_{S_b}^-)\circ (\Phi\vert_b)\circ \pi^{4,tr}_{X_b} = \pi_{X_b}^{4,tr} \ \ \ \hbox{in}\ A^4(X_b\times X_b)  \ \ \ \forall b\in B\ .\]
      In other words, the map of motives
   \[ \Phi\vert_b\colon\ \ \ttt^4(X_b)\ \to\ (S_b,\Delta^-_{S_b},1)\ \ \ \hbox{in}\ \MM_{\rm rat} \]
   has a left-inverse (given by $\Phi^{-1}\vert_b$). For very general $b\in B^\prime$ (i.e. for those $b$ such that $S_b$ has Picard number $1$), the right--hand side is just 
   the transcendental motive 
   $\ttt^2(S_b)$, and so we get a map
   \begin{equation}\label{thisone} \Phi\vert_b\colon\ \ \ttt^4(X_b)\ \to\ \ttt^2(S_b)(1)\ \ \ \hbox{in}\ \MM_{\rm rat} \end{equation}
   admitting a left-inverse (given by $\Phi^{-1}\vert_b$). For arbitrary $b\in B^\prime$, the motive $\ttt^2(S_b)(1)$ is a submotive of $(S_b,\Delta_{S_b}^-,1)$, and there is a splitting
   \begin{equation}\label{splitS} (S_b,\Delta_{S_b}^-,1)\ \cong\  \ttt^2(S_b)(1)\oplus {\mathds{1}}(1)^{\rho-1}\ \ \ \hbox{in}\ \MM_{\rm rat}\ \end{equation}
   (where $\rho$ is the Picard number of $S_b$).
   On the Verra fourfold side, there is a similar splitting
    \begin{equation}\label{splitX}   (X_b,\Delta_{X_b}^-,0)\ \cong\  \ttt^4(X_b)\oplus {\mathds{1}}(2)^{\rho-1}\ \ \ \hbox{in}\ \MM_{\rm rat}\ .\end{equation}
    The map 
     \[ \Phi^{-1}\vert_b\colon\ \  (S_b,\Delta_{S_b}^-,1)\ \to\     (X_b,\Delta_{X_b}^-,0)\ \ \ \hbox{in}\ \MM_{\rm rat}\ \]
     sends the Lefschetz part ${\mathds{1}}(1)^{\rho-1}$ of the splitting (\ref{splitS}) to the Lefschetz part ${\mathds{1}}(2)^{\rho-1}$ of the splitting (\ref{splitX})
     (if not, there is non-zero cohomology $H^2({\mathds{1}}(1)^{\rho-1})=H^2(S_b)_0\cap H^{1,1}(S_b)$ which is sent to zero under $\Phi^{-1}\vert_b$, contradicting
     the isomorphism of Proposition \ref{univ}). It follows that the map 
   (\ref{thisone}) admits a left-inverse for {\em all\/} $b\in B^\prime$.  
   
  Starting from the K3 side, things are slightly easier. That is, we now consider the relative correspondence  
    \[ \Gamma:= \Phi\circ \Delta^-_{\XX}\circ \Phi^{-1}-\Delta^-_\Ss\ \ \in\ A^2(\Ss\times_{B^\prime}\Ss)\ .\]
   We are going to use a particular instance of the ``generalized Franchetta conjecture'':
  
  \begin{theorem}[\cite{FLV}]\label{gfc} Let $\Ss\to B$ denote the universal family of smooth degree $2$ K3 surfaces. Let $\Gamma\in A^i(\Ss^{m/B})$ where $m\le 3$. The
  restriction
    \[ \Gamma\vert_{(S_b)^m}\ \ \in\ A^i\bigl((S_b)^m) \]
    is homologically trivial if and only if it is rationally trivial.
  \end{theorem}
  
  \begin{proof} This is \cite[Theorem 1.5(\rom1)]{FLV}.
  \end{proof}

Applying Theorem \ref{gfc} (with $m=2$) to $\Gamma$, we conclude that    
  \[  (\Phi\vert_b)\circ (\Phi^{-1}\vert_b)-\Delta^-_{S_b}=0 \ \ \in\ A^2(S_b\times S_b)\ ,\]  
  for all $b\in B^\prime$. Reasoning as above, this implies that $\Phi^{-1}\vert_b$ is a right-inverse to
   \begin{equation}\label{thismap}  \Phi\vert_b\colon\ \ \ (X_b,\Delta^-_{X_b},0)\ \to\  (S_b, \Delta^-_{S_b},1)\ \ \ \hbox{in}\ \MM_{\rm rat}\ ,\end{equation}
   for all $b\in B^\prime$, and so this map is an isomorphism for all $b\in B^\prime$. This already gives the required isomorphism of transcendental motives for very general $b\in B^\prime$.
   To extend this to all of $B^\prime$, we note that for any $b\in B^\prime$, we again consider the splittings (\ref{splitX}) and (\ref{splitS}).
    Under the isomorphism (\ref{thismap}) the Lefschetz motives of the splitting (\ref{splitX}) are sent to the Lefschetz motives on the splitting (\ref{splitS}) (this follows because the isomorphism of proposition \ref{univ} sends algebraic classes to algebraic classes). It follows that the isomorphism (\ref{thismap}) induces an isomorphism
            \[    \Phi\vert_b\colon\ \  \ttt^4(X_b)\ \xrightarrow{\cong}\  \ttt^2(S_b)(1)\ \ \ \hbox{in}\ \MM_{\rm rat}\ , \ \ \ \forall b\in B^\prime\ .\]
                                  
   To wrap up the proof, it only remains to extend from $B^\prime$ to $B$. Given a point on the boundary $b\in B\setminus B^\prime$, there exists a pencil $\XX\to B_1$ (i.e. $\dim B_1=1$) such that the central fibre $X_0$ is $X_b$, and a general fibre is a Verra fourfold in the restricted family $\XX\to B^\prime$. On the $K3$ surface side, we get a pencil $\Ss\to B_1$ where the central fibre $S_0$ is a K3 surface with rational double points, and the general fibre is smooth. From the Brieskorn--Tyurina theory of simultaneous resolution, we know that (up to shrinking
   $B_1$ around $0$) there exists a finite base-change $\pi\colon B_1^\heartsuit\to B_1$ such that
   the base-changed family $\Ss^\heartsuit\to B_1^\heartsuit$ has a simultaneous resolution $f\colon \wt{\Ss}^\heartsuit\to\Ss^\heartsuit$. The formalism of relative correspondences applies to the families $\wt{\Ss}^\heartsuit\to B_1^\heartsuit$ and
   $\XX^\heartsuit\to B_1^\heartsuit$.
   
   Let us define open subsets $B_1^\prime:= B_1\cap B^\prime$, and $ B_1^{\prime\,\heartsuit}:=\pi^{-1}(B_1^\prime)$.    
    The relative cycles $\Phi\in A^3(\XX\times_{B^\prime}\Ss)$ and $\Phi^{-1}\in A^3(\Ss\times_{B^\prime}\XX)$ restrict to cycles
    \[ \begin{split} &\Phi_1:= \Phi\vert_{\XX\times_{B_1^\prime}\Ss}\ \ \in\ \ A^3(\XX\times_{B_1^\prime}\Ss),\ \\
            &\Phi_1^{-1}:=\Phi^{-1}\vert_{\Ss\times_{B_1^\prime}\XX}\ \ \in\ A^3(\Ss\times_{B_1^\prime}\XX) \ .\\
            \end{split}\]
   We can pull-back these cycles along $\pi\circ f$, and then extend over all of $B_1^\heartsuit$, i.e. there exist (non-canonical) cycles
          \[ \begin{split} &{}^\heartsuit\bar{\Phi}_1\ \ \in\ \ A^{3}(\XX^\heartsuit\times_{B_1^\heartsuit}\wt{\Ss}^\heartsuit),\ \\
            &{}^\heartsuit \bar{\Phi}_1^{-1}\ \ \in\ A^{3}(\wt{\Ss}^\heartsuit\times_{B_1^\heartsuit}\XX^\heartsuit) \ \\
            \end{split}\]
     restricting to $f^\ast\pi^\ast(\Phi_1)$ resp. to $f^\ast\pi^\ast(\Phi_1^{-1})$ (NB: by abuse of language, we will simply write $\pi$ for all morphisms induced by $\pi$).
     
     By construction, for general $b\in B_1^\heartsuit$ the fibrewise restrictions of these cycles induce mutually inverse isomorphisms of motives
     \[  \begin{split} &{}^\heartsuit\bar{\Phi}_1\vert_b\colon\ \   \ttt^4(X_b)\ \xrightarrow{\cong}\ \ttt^2(\wt{S}_b)(1)    ,\ \\
            &{}^\heartsuit\bar{\Phi}_1^{-1}\vert_b\colon\ \   \ttt^2(\wt{S}_b)(1)\ \xrightarrow{\cong}\ \ttt^4(X_b)\ \ \ \ \ \ \hbox{in}\ \MM_{\rm rat}\ .   \\
            \end{split}\]     
But then, the extension principle for families of cycles \cite[Lemma 3.2]{Vo}, \cite[Proposition 2.4]{V11} implies that the same is true for {\em all\/} $b\in B_1^\heartsuit$.

Let us now pick a point $b_0\in B_1^\heartsuit$ such that $\pi(b_0)=0$ in $B_1$, and define correspondences
  \[  \begin{split}  &\Phi_0:= \pi_\ast (f_{b_0})_\ast\Bigl( {}^\heartsuit\bar{\Phi}_1\vert_{b_0}\Bigr) \ \ \in\ A^3(X_0\times {S}_0)\ ,\\
                           &\Phi_0^{-1}:= \pi_\ast (f_{b_0})_\ast \Bigl( {}^\heartsuit\bar{\Phi}_1^{-1}\vert_{b_0}\Bigr) \ \ \in\ A^3({S}_0\times X_0)\ .\\
                         \end{split}\]
%

Let us consider the diagram
  \[ \begin{array}[c]{ccc}
         \ttt^4(X_{b_0}) & \xrightarrow{    {}^\heartsuit\bar{\Phi}_1\vert_{b_0}} & \ttt^2(\wt{S}_{b_0})(1) \\
      &&\\
    \ \ \ \   \uparrow{\scriptstyle {}^t  \Gamma_\pi}   && \ \ \ \   \uparrow{\scriptstyle {}^t  \Gamma_\pi}   \\
    &&\\
     \ttt^4(X_{0}) &   \dashrightarrow    & \ \ttt^2(S_{0})(1) \ .\\
     \end{array}\]
     (For the motive of the singular surface $S_0$, cf. Subsection \ref{ss:alex}.)
     As we have just seen, the upper horizontal arrow is an isomorphism. The vertical arrows are isomorphisms, because the transcendental motive is a birational invariant. For the same reason, if $S\to S_0$ denotes a minimal resolution of singularities, there is an isomorphism $\ttt^2(S_0)\cong \ttt(S)$.
     As a result, we get an isomorphism
     \[ \ttt^4(X_0)\ \cong\ \ttt^2(S)(1)\ \ \ \hbox{in}\ \MM_{\rm rat}\ .\]
   This closes the proof.\footnote{The last step in the proof is somewhat delicate and round-about. This is due to the fact that the formalism of relative correspondences only exists for smooth morphisms between smooth quasi-projective varieties. Hence, if we extend the correspondences $\Phi_1, \Phi_1^{-1}$ over all of $B_1$, we cannot make sense of their composition.}                                                 
  \end{proof} 
  
  \begin{remark}\label{later} It follows from Theorem \ref{verrak3} (by taking cohomology of the motives) that for every Verra fourfold $X$, with associated K3 surface $S$, there is an isomorphism
    \[ H^4_{tr}(X)\ \cong\ H^2_{tr}(S)(1)\ .\]
  Using a simultaneous resolution of singularities (as in the above proof), one can check this isomorphism is compatible with bilinear forms up to a non-zero coefficient.  
    \end{remark}

 \subsection{MCK decomposition} 
 The motivic relation between Verra fourfolds and K3 surfaces (Theorem \ref{verrak3}) allows us to establish an MCK decomposition:
  
  \begin{theorem}\label{verramck} Let $X$ be a Verra fourfold. Then $X$ has a self-dual MCK decomposition.
  \end{theorem}
  
  \begin{proof} As above, let $\XX\to B$ denote the universal family of Verra fourfolds. One can define a ``relative K\"unneth decomposition'' 
    \[ \pi^j_\XX\in A^4(\XX\times_B \XX)\ ,\ \ \ j=0,\ldots,8\ ,\]
    by defining $\pi^j_\XX$ for $j<4$ as the restriction of certain completely decomposed correspondences on $\bar{C}\times\bar{C}$ (where $\bar{C}$ is as before the cone over $\PP^2\times\PP^2$).
    Concretely, this means that 
    \[ \pi^0_\XX\vert_b= {1\over d}\, h^4\times X_b\ \ \ \hbox{in}\ A^4(X_b\times X_b)\ \]
    for some $d\in\NN$, where $h$ is a hyperplane section, and
    \[ \pi^2_\XX\vert_b= \ell_1\times\ell_1^\vee +  \ell_2\times\ell_2^\vee   \ \ \ \hbox{in}\ A^4(X_b\times X_b)\ ,\] 
    where the $\ell_i$ are a basis for $H^6(X_b)$, and the $\ell_i^\vee$ are a dual basis for $H^2(X_b)$.
    
    The other components are defined as $\pi^j_\XX={}^t \pi^{8-j}_\XX$ for $j>4$, and finally 
    \[ \pi^4_\XX:=\Delta_\XX-\sum_{j\not=4} \pi^j_\XX\ \ \ \in\ A^4(\XX\times_B \XX)\ .\]
    It is readily checked that the fibrewise restriction $\pi^j_\XX\vert_b$ is a self-dual Chow--K\"unneth decomposition, for every $b\in B$. 
    
    Let us check that this fibrewise restriction is {\em multiplicative\/}.
   Let
     \[ \Delta_\XX^{sm}:= (p_{12})^\ast(\Delta_XX)\cdot (p_{23})^\ast(\Delta_\XX)\ \ \ \in\ A^8(\XX\times_B \XX\times_B \XX) \]
     be the ``relative small diagonal''.
      Given a triple $(i,j,k)\in\NN^3$, we can form the composition of relative correspondences
     \[ \Gamma_{ijk}:= \pi_\XX^i\circ \Delta_\XX^{sm}\circ \bigl( (p_{13})^\ast (\pi_\XX^j)\cdot (p_{24})^\ast(\pi_\XX^k)\bigr)\ \ \ \in\ A^8(\XX\times_B \XX\times_B \XX)\ . \]    
     To establish multiplicativity of the fibrewise CK decomposition, we need to prove the fibrewise vanishing
      \begin{equation}\label{need} \Gamma_{ijk}\vert_{(X_b)^3}\stackrel{??}{=}0\ \ \ \hbox{in}\ A^8(X_b\times X_b\times X_b)\ \ \ \forall b\in B\ ,\ \forall i\not=j+k\ .\end{equation}
    
    Restricting to a smaller open $B^\prime\subset B$, there is a universal family $\Ss\to B^\prime$ of associated degree $2$ K3 surfaces, and a correspondence $\Phi\in A^3(\XX\times_{B^\prime} \Ss)\oplus \bigoplus A^\ast(\XX)$, inducing a fibrewise inclusion as submotive
    \[ \Phi_b\colon\ \ h(X_b)\ \hookrightarrow\ h(S_b)(1)\oplus \bigoplus{\mathds{1}}(\ast)\ \ \ \hbox{in}\ \MM_{\rm rat}\ \]
    (i.e., fibrewise there exists a left-inverse), cf. Theorem \ref{verrak3}.

%
%
    
    Let us now consider the ``product arrow'' (in the category $\MM_{\rm rat}^{/{B^\prime}}$ of relative Chow motives over $B^\prime$, cf. \cite[Chapter 8]{MNP})
    \[ \begin{split} (\Phi,\Phi,\Phi)\colon\ \  h(\XX^{3/{B^\prime}})\ \to\ h(\Ss^{3/{B^\prime}})(3)\oplus &\bigoplus h(\Ss^{2/{B^\prime}})(\ast)\oplus \\   
                &\bigoplus h(\Ss)(\ast)  \oplus \bigoplus h(B^\prime)(\ast)\ \ \ \hbox{in}\ \MM_{\rm rat}^{/{B^\prime}}\ .\\\end{split}\]
     Since fibrewise there exists a left-inverse, the fibrewise induced maps on Chow groups are injective: for all $b\in B^\prime$, one has injections
     \begin{equation}\label{inj}  (\Phi_b,\Phi_b,\Phi_b)_\ast\colon\ \  A^\ell\bigl((X_b)^3\bigr)\ \hookrightarrow\ A^{\ell-3}\bigl((S_b)^3\bigr)\oplus \bigoplus A^\ast\bigl((S_b)^2\bigr)\oplus \\   
                \bigoplus A^\ast(S_b)  \oplus \bigoplus \QQ\ .\end{equation}
Let us now pick up our relative correspondence $\Gamma_{ijk}$ with $i\not=j+k$, and restrict it to $A^8(\XX^{3/{B^\prime}})$. Under the assumption $i\not=j+k$, we have that
$\Gamma_{ijk}$ is fibrewise homologically trivial.
Hence, the image
  \[  (\Phi,\Phi,\Phi)_\ast (\Gamma_{ijk})\ \ \in\ A^\ast(\Ss^{3/{B^\prime}})\oplus \bigoplus A^\ast(\Ss^{2/{B^\prime}}) \oplus \bigoplus A^\ast(\Ss)\oplus \bigoplus \QQ \]
  is also fibrewise homologically trivial. But then, Theorem \ref{gfc} implies that the image is fibrewise rationally trivial:
  \[  (\Phi,\Phi,\Phi)_\ast (\Gamma_{ijk})\vert_b=0\ \ \ \hbox{in}\ A^\ast\bigl((S_b)^3\bigr)\oplus \bigoplus A^\ast\bigl((S_b)^2\bigr)\oplus \\   
                \bigoplus A^\ast(S_b)  \oplus \bigoplus \QQ\ ,\]
                for all $b\in B^\prime$.
       Now, in view of the injectivity (\ref{inj}), we may conclude that
       \[ \Gamma_{ijk}\vert_b=0\ \ \ \hbox{in}\ A^8\bigl((X_b)^3\bigr)\ \ \ \forall b\in B^\prime\ .\]
   An application of \cite[Lemma 3.2]{Vo} allows to extend this to the larger base $B$:
    \[     \Gamma_{ijk}\vert_b=0\ \ \ \hbox{in}\ A^8\bigl((X_b)^3\bigr)\ \ \ \forall b\in B\ ,\]                
    and so we have proven the required vanishing (\ref{need}).              
   \end{proof}

  We record an observation for later use:
  
  \begin{lemma}\label{pure0} Let $X$ be a Verra fourfold, let $S$ be an associated K3 surface, and let
    \[ \Psi\colon\ \ h(X)\ \to\  h(S)(1)\oplus \bigoplus {\mathds{1}}(\ast)\ \ \ \hbox{in}\ \MM_{\rm rat} \]
    be the map of motives that has a left-inverse (cf. Theorem \ref{verrak3}).
    Then $\Psi$ is of pure grade $0$, i.e. $\Psi\in A^3_{(0)}(X\times S)\oplus A^\ast_{(0)}(X)$.
   \end{lemma}
   
   \begin{proof} This is an easy corollary of the generalized Franchetta type result (Theorem \ref{gfc}). Indeed, let $\pi^j_\XX$ be the ``relative K\"unneth decomposition'' as in the proof of 
   Theorem \ref{verramck}. The universal family $\Ss\to B$ of smooth genus $2$ K3 surfaces also has a ``relative K\"unneth decomposition'' $\pi^j_\Ss\in A^2(\Ss\times_B \Ss)$. This induces a
   ``relative K\"unneth decomposition'' 
     \[ \pi^j_{\XX\times_B \Ss}:= \sum_{k+\ell=j} (p_\XX)^\ast \pi^k_\XX \cdot (p_\Ss)^\ast \pi^\ell_\Ss\ \ \ \in A^6\bigl( (\XX\times_B \Ss)\times_B (\XX\times_B \Ss)\bigr)\ ,\]
     where $p_\XX, p_\Ss$ are the obvious projections from $\XX\times_B \Ss\times_B \XX\times_B \Ss$ to $\XX\times_B \XX$ resp. to $\Ss\times_B \Ss$.
     
    As we have seen, the correspondence labelled $\Psi$ in the statement of the lemma is the restriction to a fibre of a relative correspondence
     \[ \Psi=(\Psi_\Ss, \Psi_0,\ldots,\Psi_4)\ \ \in\ A^3(\XX\times_B \Ss)\oplus \bigoplus_{i=0}^4 A^i(\XX)\ .\]
    We consider the relative cycle
     \[  \Gamma_j:= (\pi^j_{\XX\times_B \Ss})_\ast (\Psi_\Ss)  \ \ \in\ A^3(\XX\times_B \Ss)\ ,\ \ \ j\not=6\ .\]
     Clearly, the relative cycle $\Gamma_j$ is fibrewise homologically trivial for $j\not=6$. It follows that
     \[ (\Psi_{},\Delta_\Ss)_\ast (\Gamma_j)\ \ \in\ A^2(\Ss\times_B \Ss)\oplus \bigoplus A^\ast(\Ss) \]
     is also fibrewise homologically trivial. But then Theorem \ref{gfc} (with $m=2$) implies that 
     \[ \Bigl((\Psi_{},\Delta_\Ss)_\ast (\Gamma_j)\Bigr)\vert_b=0\ \ \in\ A^2(\Ss\times_B \Ss)\oplus \bigoplus A^\ast(\Ss)\ ,\ \ \ \forall b\in B\ . \]
     But as $\Psi$ has a fibrewise left-inverse, the map
     \[   (\Psi_{}\vert_b,\Delta_\Ss\vert_b)_\ast   \colon\ \ A^\ast(X_b\times S_b)\ \to\ A^{\ast-1}(S_b\times S_b)\oplus \bigoplus A^\ast(S_b) \]
     is injective, and so
     \[ (\Gamma_j)\vert_b=0\ \ \hbox{in}\ A^3(X_b\times S_b)\ \ \ \forall b\in B\ .\]
     This means that 
       \[  (\Psi_\Ss)\vert_b =(\pi^6_{X_b\times S_b})_\ast (\Psi_\Ss)\vert_b\ \ \ \hbox{in}\ A^3(X_b\times S_b)\ ,\]    
       i.e. $(\Psi_\Ss)\vert_b$ is of pure grade $0$.
     
     The same reasoning, using theorem \ref{gfc} with $m=1$, gives that the $(\Psi_i)\vert_b$, $i=0,\ldots,4$ are of pure grade $0$. We conclude that
     \[ \Psi\vert_b= (\Psi_\Ss, \Psi_0,\ldots,\Psi_4)\vert_b \ \ \ \in\ A^3(X_b\times S_b)\oplus A^\ast(X_b) \]
     is of pure grade $0$ for all $b\in B$, which is what we wanted to prove.     
   \end{proof}

  \section{Corollaries}

 \subsection{The Chow ring of $X$}
    
    \begin{definition}\label{can} Let $X$ be a Verra fourfold, and let $\pi\colon X\to\PP^2\times\PP^2$ denote the double cover. We define two curve classes
    \[  \begin{split}  \ell_1&:= \pi^\ast(h\times h^2)\ ,\\
                              \ell_2&:=\pi^\ast(h^2\times h)\ \ \ \in\ A^3(X)\ ,\\
                             \end{split} \]
                         where $h\in A^1(\PP^2)$ is the class of a hyperplane.     
       \end{definition}
    
    \begin{remark} Clearly, the classes $\ell_1,\ell_2$ generate $H^6(X)\cong \QQ^2$.
    \end{remark}

  \begin{corollary}\label{subring} Let $X$ be a Verra fourfold. The subalgebra
   \[ R^\ast(X):= < A^1(X)_{\QQ}, A^2(X)_{\QQ}, c_j(T_X), \ell_1, \ell_2>\ \ \ \subset\ A^\ast(X)_{\QQ}\ \]
   injects into cohomology via the cycle class map.  
  \end{corollary}  
    
   \begin{proof} It follows from Theorem \ref{verramck} (via Definition \ref{mck}) that the Chow ring of $X$ admits a bigrading $A^\ast_{(\ast)}(X)$.
   Since $X$ is rationally connected, one has $A^j_{hom}(X)=0$ for $j\le 2$. Since clearly $(\pi_X^k)_\ast A^j(X)\subset A^j_{hom}(X)$ for all $k\not= 2j$, it follows that
     \[ A^j(X)=A^j_{(0)}(X)\ \ \ \forall j\le 2\ .\]
    It remains to show that the Chern class $c_3(T_X)$ is in $A^3_{(0)}(X)$. To this end, we consider again the universal family $\XX\to B$ of Verra fourfolds. The MCK decomposition obviously exists relatively, i.e. we have relative cycles 
    \[ \pi_\XX^j\in A^4(\XX\times_B \XX)\ ,\ \ j=0,\ldots,8\ ,\]
    restricting to an MCK decomposition on each fibre $X_b$.
    (Indeed, the $\pi_{X_b}^j$ for $j\not=4$ are defined in terms of cycles coming from (the cone over) $\PP^2\times\PP^2$, and the remaining component $\pi_{X_b}^4$ is obtained by subtraction from the diagonal.) 
    
    Let us now consider the relative cycle
    \[ \Gamma:=  (\pi_\XX^j)_\ast c_3(T_{\XX/B})\ \ \ \in A^3(\XX)\ , \ \ j\not= 6\ .\]
   Since $j\not=6$, the restriction $\Gamma\vert_{X_b}$ is homologically trivial for each $b\in B$.
   
   Restricting to a smaller open $B^\prime\subset B$, there exists a universal family of associated K3 surfaces $\Ss\to B$, and a relative correspondence $\Phi\in A^3(\XX\times_{B^\prime} \Ss)$ as in Theorem \ref{verrak3}. The relative cycle
   \[ \Phi_\ast(\Gamma)\ \ \in\  A^2(\Ss)  \]
   is fibrewise homologically trivial (since $\Gamma$ is fibrewise homologically trivial). Theorem \ref{gfc} then implies that $\Phi_\ast(\Gamma)$ is fibrewise rationally trivial:
   \[  (\Phi_b)_\ast(\Gamma\vert_{X_b})= \bigl(  \Phi_\ast(\Gamma)\bigr)\vert_b=0\ \ \hbox{in}\  A^2(S_b)\oplus   \bigoplus A^0(\hbox{point})\ .  \]  
   But we have seen (Theorem \ref{verrak3}) that
   \[ (\Phi_b)_\ast\colon\ \ A^3_{hom}(X_b)\ \to\ A^2_{hom}(S_b)\oplus \bigoplus \QQ \]
   is injective, and so $\Gamma\vert_{X_b}=0$ in $A^3(X_b)$ for all $b\in B$. This shows that $c_3(T_{X_b})\in A^3_{(0)}(X_b)$ for all $b\in B$.
   
  The same argument also proves that $\ell_1,\ell_2\in A^3_{(0)}(X)$. Since $A^\ast_{(0)}(X)\subset A^\ast(X)$ is a subring, we may thus conclude that
    \[ R\ \subset\ A^\ast_{(0)}(X)\ .\]
    It remains to show that the cycle class map induces an injection
    \[ A^\ast_{(0)}(X)\ \to\ H^\ast(X) \ .\]
    This is only non-trivial for $A^3_{(0)}(X):= (\pi^6_X)_\ast A^3(X)$. However, the K\"unneth component $\pi^6_X$ is supported on $X\times C$ where $C\subset X$ is a union of irreducible curves $C=\cup C_j$. The action of $\pi^6_X$ on $A^3_{(0)}(X)\cap A^3_{hom}(X)$ (which is the identity) factors over $\oplus A^0_{hom}(C_j)=0$. It follows that
     \[ A^3_{(0)}(X)\cap A^3_{hom}(X)=0\ ,\]
     as requested.   
   \end{proof}

 \begin{remark} The curve classes $\ell_1,\ell_2\in A^3(X)$ play the role of the Beauville--Voisin distinguished $0$-cycle $\mathfrak{o}_S={1\over 24}c_2(T_S)\in A^2(S)$ on a K3 surface $S$ \cite{BV}. Indeed, Corollary \ref{subring} implies that 
   \[ A^1(X)\cdot A^2(X)\ \ \in\ A^3_{(0)}(X)=\QQ[\ell_1]\oplus \QQ[\ell_2]\ \ \subset A^3(X)\ .\]
  \end{remark}
   
 \begin{remark} One can add more cycles to the subring $R^\ast(X)$ of Corollary \ref{subring}. For example, let $D\subset X$ denote the ramification locus of the double cover $\pi\colon X\to\PP^2\times\PP^2$. Then $D$ is isomorphic to the Verra threefold $Y\subset\PP^2\times\PP^2$ (i.e., the branch locus). Hence,
  \[ \begin{split} \ima\bigl( A_1(D)\to A^3(X)\bigr) &= \ima\bigl( A_1(Y)\to A^3(X)\bigr)\\
  &\subset \ima \bigl( A^3(\PP^2\times\PP^2)\to A^3(X)\bigr)=\QQ[\ell_1]\oplus \QQ[\ell_2]=A^3_{(0)}(X)\ ,\\
  \end{split}\]
  and so $ \ima\bigl( A_1(D)\to A^3(X)\bigr)$ may be added to the subring $R^\ast(X)$.
  
   The divisor $D\subset X$ has behaviour akin to that of {\em constant cycle subvarieties\/} \cite{Huy}.
   
 \end{remark}

 \subsection{The Chow ring of $X^m$}
 
 More generally, one can consider the Chow ring of self-products of $X$:
 
  \begin{corollary}\label{corm} Let $X$ be a Verra fourfold, and let $m\in\NN$. Let $R^\ast(X^m)\subset A^\ast(X^m)$ be the $\QQ$-subalgebra
   \[ R^\ast(X^m):=  \langle (p_i)^\ast A^1(X), (p_i)^\ast A^2(X), (p_{ij})^\ast(\Delta_X), (p_i)^\ast c_\ell(T_X)\rangle\ \ \ \subset\ A^\ast(X^m)\ .\]
   (Here $p_i\colon X^m\to X$ and $p_{ij}\colon X^m\to X^2$ denote projection to the $i$th factor, resp. to the $i$th and $j$th factor.)
   
  The cycle class map induces injections
   \[ R^j(X^m)\ \hookrightarrow\ H^{2j}(X^m)\]
   in the following cases:
   
   \begin{enumerate}
      
   \item $m=2$ and $j\ge 5$;
   
   \item $m=3$ and $j\ge 9$.
   \end{enumerate}
       \end{corollary}
 
 \begin{proof} Theorem \ref{main}, in combination with Proposition \ref{product}, ensures that $X^m$ has an MCK decomposition, and so $A^\ast(X^m)$ has the structure of a bigraded ring under the intersection product. The corollary is now implied by the combination of the two following claims:

\begin{claim}\label{c1} There is inclusion
  \[ R^\ast(X^m)\ \ \subset\ A^\ast_{(0)}(X^m)\ .\]
  \end{claim}
  
 \begin{claim}\label{c2} The cycle class map induces injections
   \[ A^j_{(0)}(X^m)\ \hookrightarrow\ H^{2j}(X^m)\ \] 
   provided $m=2$ and $j\ge 5$, or $m=3$ and $j\ge 9$.
\end{claim}

To prove Claim \ref{c1}, we note that $A^k(X)=A^k_{(0)}(X)$ for $k\not=3$, and $c_3(T_X)$ is in $A^3_{(0)}(X)$ (Corollary \ref{subring}), and so all Chern classes $c_\ell(T_X)$ are in $A^\ast_{(0)}(X)$. The fact that $\Delta_X\in A^4_{(0)}(X\times X)$ is a general fact for any $X$ with a self-dual MCK decomposition \cite[Lemma 1.4]{SV}. Since the projections $p_i$ and $p_{ij}$ are pure of grade $0$ \cite[Corollary 1.6]{SV}, and 
$A^\ast_{(0)}(X^m)$ is a ring under the intersection product, this proves Claim \ref{c1}.

To prove Claim \ref{c2}, we observe that Manin's blow-up formula \cite[Theorem 2.8]{Sc} gives an isomorphism of motives
  \[ \hh(X)\cong \hh(S)(1)\oplus {\mathds{1}} \oplus  {\mathds{1}}(1)^{} \oplus {\mathds{1}}(2)^{}   \oplus  {\mathds{1}}(3)^{} \oplus  {\mathds{1}}(4)\ \ \ \hbox{in}\ \MM_{\rm rat}\ .\]
  Moreover, we know from Lemma \ref{pure0} that the correspondence inducing this isomorphism is of pure grade $0$.
  
 In particular, taking Chow groups on both sides it follows that for any $m\in\NN$ we have isomorphisms 
  \[ A^j(X^m)\cong A^{j-m}(S^m)\oplus \bigoplus_{k=0}^4  A^{j-m+1-k}(S^{m-1}) \oplus \bigoplus A^\ast(S^{m-2})\oplus \bigoplus_{\ell\ge 3} A^\ast(S^{m-\ell})  \ ,  \] 
  and this isomorphism respects the $A^\ast_{(0)}()$ parts. Claim \ref{c2} now follows from the fact that for any surface $S$ with an MCK decomposition, and any $m\in\NN$, the cycle class map induces injections
  \[ A^i_{(0)}(S^m)\ \hookrightarrow\ H^{2i}(S^m)\ \ \ \forall i\ge 2m-1\ \]
  (this is noted in \cite[Introduction]{V6}, cf. also \cite[Proof of Lemma 2.20]{acs}).
  \end{proof}

  \subsection{Motives of isogenous Verra fourfolds}
  
  \begin{corollary}\label{isog} Let $X,X^\prime$ be two Verra fourfolds that are isogenous (i.e., there exists an isomorphism of $\QQ$-vector spaces $H^4(X,\QQ)\cong H^4(X^\prime,\QQ)$ compatible with the Hodge structures and with cup product). Then there is an isomorphism of Chow motives
    \[ \hh(X)\ \xrightarrow{\cong}\ \hh(X^\prime)\ \ \ \hbox{in}\ \MM_{\rm rat}\ .\]
    \end{corollary}
    
    \begin{proof} The assumption implies (actually, is equivalent to) the existence of a Hodge isometry 
      \[ H^4_{tr}(X)\ \cong\ H^4_{tr}(X^\prime)\ .\] 
      Let $S, S^\prime$ be two K3 surfaces associated to $X$ resp. $X^\prime$. Then, because the isomorphism $H^4_{tr}(X)\cong H^2_{tr}(S)(1)$ is compatible with Hodge structure and with cup product (up to some universal non-zero coefficient (cf. Remark \ref{later})), there is also a Hodge isometry 
    \[ H^2_{tr}(S)\ \cong\ H^2_{tr}(S^\prime)\ .\]
    The celebrated Huybrechts result \cite[Theorem 0.2]{Huy2} then gives the existence of an isomorphism of Chow motives $\hh(S)\cong \hh(S^\prime)$, and so in particular an isomorphism of Chow motives $\ttt^2(S)\cong \ttt^2(S^\prime)$. One concludes by applying Theorem \ref{verrak3}. 
      \end{proof}
   
   One has the following variant of Corollary \ref{isog}:
   
   \begin{corollary}\label{isogvar} Let $X$ be a Verra fourfold, let $T$ be a K3 surface and assume there exists a Hodge isometry $H^4_{tr}(X,\QQ)\cong H^2_{tr}(T,\QQ)(1)$ 
   Then there is an isomorphism of Chow motives
   \[ \ttt^4(X)\ \xrightarrow{\cong}\ \ttt^2(T)(1)\ \ \ \hbox{in}\ \MM_{\rm rat}\ .\]
   \end{corollary}
   
   \begin{proof} Let $S$ be a K3 surface associated to $X$. As above, the Abel--Jacobi isomorphism (cf. remark \ref{later}) gives a Hodge isometry
         \[ H^2_{tr}(S)\ \cong\ H^2_{tr}(T)\ .\] 
         Then Huybrechts' result \cite{Huy2} allows to conclude there is an isomorphism of Chow motives $\ttt^2(S)\cong \ttt^2(T)$. 
         Combined with the isomorphism of Theorem \ref{verrak3}, this proves the corollary.         
           \end{proof}   
   
   \begin{remark} The argument of Corollary \ref{isogvar} shows in particular the remarkable fact that any Hodge isometry between a Verra fourfold and a K3 surface is induced by an algebraic cycle. The same statement is known (but considerably more difficult to prove !) for integral Hodge isometries between cubic fourfolds and K3 surfaces \cite[Theorem 1.3]{AT}.
    \end{remark}

 \subsection{Decomposition in the derived category}
  For a smooth projective morphism $\pi\colon \XX\to B$, Deligne \cite{Del} has proven a decomposition in the derived category of sheaves of $\QQ$-vector spaces on $B$:
    \begin{equation}\label{del}  R\pi_\ast\QQ\cong \bigoplus_i  R^i \pi_\ast\QQ[-i]\ .\end{equation}
    As explained in \cite{Vdec}, for both sides of this isomorphism there is a cup product: on the right-hand side, this is the direct sum of the usual cup-products of local systems, while on the left-hand side, this is the derived cup product (inducing the usual cup product in cohomology). In general, the isomorphism (\ref{del}) is {\em not\/} compatible with these cup-products, even after shrinking the base $B$ (cf. \cite{Vdec}). In some rare cases, however, there is such a compatibility (after shrinking): this is the case for families of abelian varieties \cite{DM}, and for families of K3 surfaces \cite{Vdec} (cf. also \cite[Theorem 4.3]{V6} and \cite[Corollary 8.4]{FTV} for some further cases).
    
   Given the close link to K3 surfaces, it is not surprising that Verra fourfolds also have such a multiplicative decomposition: 
   
 \begin{corollary}\label{deldec} Let $\XX\to B$ be a family of Verra fourfolds. There is a non-empty Zariski open $B^\prime\subset B$, such that the isomorphism (\ref{del}) becomes multiplicative after shrinking to $B^\prime$.
 \end{corollary}
 
 \begin{proof} This is a formal consequence of the existence of a relative MCK decomposition, cf. \cite[Proof of Theorem 4.2]{V6} and \cite[Section 8]{FTV}.
 \end{proof}
   
  This has the following concrete consequence, which is similar to a result for families of $K3$ surfaces obtained by Voisin \cite[Proposition 0.9]{Vdec}:
  
  \begin{corollary}\label{deldec2} Let $\XX\to B$ be a family of Verra fourfolds.   
 Let $z\in A^r(\XX^{m/B})$ be a polynomial in (pullbacks of) divisors and codimension $2$ cycles on $\XX$.  
  Assume the fibrewise restriction $z\vert_b$ is homologically trivial, for some $b\in B$. Then there exists a non-empty Zariski open $B^\prime\subset B$
  such that
    \[ z=0\ \ \ \hbox{in}\ H^{2r}\bigl((\XX^\prime)^{m/B^\prime},\QQ\bigr)\ .\]
    \end{corollary}
    
    \begin{proof} The argument is the same as \cite[Proposition 0.9]{Vdec}. First, one observes that divisors $d_i$ and codimension $2$ cycles $e_j$ on $\XX$ admit a cohomological decomposition (with respect to the Leray spectral sequence)
        \[ \begin{split}  d_i&= d_{i0} + \pi^\ast(d_{i2})\ \ \ \hbox{in}\ H^0(B, R^2\pi_\ast\QQ)\oplus \pi^\ast H^2(B,\QQ) \cong H^2(\XX,\QQ)\ ,\\
                                e_j&= e_{j0} +\pi^\ast(e_{j2}) +\pi^\ast(e_{j4})\ \ \ \hbox{in}\  H^0(B, R^4\pi_\ast\QQ)\oplus \pi^\ast H^2(B)^{\oplus 2}  \oplus \pi^\ast H^4(B) \cong H^4(\XX,\QQ)\ .\\  
                        \end{split}\]
      We claim that the cohomology classes $d_{ik}$ and $e_{jk}$ are {\em algebraic\/}. This claim implies the corollary: indeed, given a polynomial
      $z=p(d_i,e_j)$, one may take $B^\prime$ to be the complement of the support of 
      the cycles $d_{i2}$, $e_{j2}$ and $e_{j4}$. Then over the restricted base $B^\prime$ one has equality
      \[  z:=p(d_i,e_j)= p(d_{i0},e_{j0})\ \ \ \hbox{in}\ H^{2r}\bigl((\XX^\prime)^{m/B^\prime},\QQ\bigr)\ .\]
      Multiplicativity of the decomposition ensures that (after shrinking the base some more)
      \[  p(d_{i0},e_{j0})\ \ \in\ H^0(B^\prime, R^{2r}(\pi^m)_\ast\QQ)\ \ \subset\ H^{2r}\bigl((\XX^\prime)^{m/B^\prime},\QQ\bigr)\ ,\]
      and so the conclusion follows.
      
      The claim is proven for divisor classes $d_i$ in \cite[Lemma 1.4]{Vdec}. For codimension $2$ classes $e_j$, the argument is similar to loc. cit.:
      let $h\in H^2(\XX)$ be an ample divisor class, and let $h_0$ be the part that lives in $H^0(B,R^2\pi_\ast\QQ)$. One has
      \[ e_j (h_0)^4 =  e_{j0}  (h_0)^4 +\pi^\ast(e_{j2})  (h_0)^4+\pi^\ast(e_{j4})  (h_0)^4\ \ \ \hbox{in}\ H^{12}(\XX,\QQ)\ .\]
      By multiplicativity, after some shrinking of the base the first two summands are contained in $H^0(B^\prime, R^{12}\pi_\ast\QQ)$, resp. in $H^2(B^\prime, R^{10}\pi_\ast\QQ)$, hence they are zero as $\pi$ has $4$-dimensional fibres. The above equality thus simplifies to
      \[ e_j (h_0)^4 =    \pi^\ast(e_{j4})  (h_0)^4\ \ \ \hbox{in}\ H^{12}(\XX,\QQ)\ .\]
      Pushing forward to $B^\prime$, one obtains
      \[ \pi_\ast (e_j (h_0)^4)= \pi_\ast \bigl((h_0)^4\bigr) e_{j4} = \lambda\, e_{j4}\ \ \ \hbox{in}\ H^4(B^\prime)\ ,\]
      for some $\lambda\in\QQ^\ast$. As the left-hand side is algebraic, so is $e_{j4}$.
      
      Next, one considers
       \[  e_j (h_0)^3 =  e_{j0}  (h_0)^3 +\pi^\ast(e_{j2})  (h_0)^3+\pi^\ast(e_{j4})  (h_0)^3\ \ \ \hbox{in}\ H^{10}(\XX,\QQ)\ .\]      
       The first summand is again zero for dimension reasons, and so
       \[ \pi^\ast(e_{j2})  (h_0)^3 =    e_j (h_0)^3 -  \pi^\ast(e_{j4})  (h_0)^3\ \ \in\ H^{10}(\XX,\QQ)\ \]
       is algebraic. A fortiori, $\pi^\ast(e_{j2})  (h_0)^4$ is algebraic, and so
       \[ \pi_\ast ( \pi^\ast (e_{j2})  (h_0)^4) =    \pi_\ast \bigl((h_0)^4\bigr) e_{j2} = \mu\, e_{j2}\ \ \ \hbox{in}\ H^2(B^\prime)\ ,  \ \ \ \mu\in\QQ^\ast\ ,\]          
       is algebraic.
      \end{proof}
 
  \subsection{Finite-dimensionality, generalized Hodge conjecture}
  
  \begin{corollary}\label{other} Let $X$ be a Verra fourfold, and assume $\dim H^{2,2}(X,\QQ)\ge 20$.
  
  \begin{enumerate}[(i)]

\item
   $X$ has finite-dimensional motive, in the sense of Kimura \cite{Kim}.
  
 \item 
   The generalized Hodge conjecture is true for $X^m$, $m\in\NN$.
    \end{enumerate}
     \end{corollary}
  
  \begin{proof} The assumption on $X$ is equivalent to $\dim H^4_{tr}(X)\le 3$. This means that any of the two associated K3 surfaces $S$ has $\dim H^2_{tr}(S)\le 3$, and so
  the Picard number of $S$ is $\ge 19$. Since K3 surfaces of Picard number $19$ or $20$ are either Kummer surfaces, or are related to Kummer surfaces via a Shioda--Inose structure, they have finite-dimensional motive \cite{Ped}. In view of Theorem \ref{verrak3}, this implies that $X$ also has finite-dimensional motive.
  
  As for (\rom2), the above can be restated as saying that $X$ is {\em motivated\/} by an abelian surface, in the sense of \cite{Ara}. The generalized Hodge conjecture is known for self-products of abelian surfaces \cite[7.2.2]{Ab}, \cite[8.1(2)]{Ab2}. This implies (via the argument of \cite{Ara}) the generalized Hodge conjecture for self-products of $X$.
   \end{proof}
   
  \begin{remark} It is worth pointing out that for Verra fourfolds as in Corollary \ref{other}, the associated double EPW quartic $Z$ also has finite-dimensional motive. This follows from 
  \cite[Theorem 0.1]{Bul}, in view of the fact that $Z$ is a moduli space of twisted stable sheaves on a K3 surface $S$ with $\rho(S)\ge 19$. 
  This argument actually shows that if the Verra fourfold has finite-dimensional motive, then also the associated double EPW quartic has finite-dimensional motive. In order to explain this ``coincidence'', it would be interesting to find a motivic relation between $X$ and $Z$, similar to the relation between a cubic fourfold and its Fano variety of lines \cite{22}.
   \end{remark}

\vskip1cm
\begin{nonumberingt} Thanks to a referee for helpful comments. Many thanks to papa and Ute for hospitably receiving me in sunny Breskens (June 2018), where this note was conceived. Many thanks to Jiji and Baba for their hospitality in mushiatsui Kokury\={o} (August 2018), where this note was completed.
\end{nonumberingt}

\end{document}